\documentclass[a4paper,twoside,10pt]{article}

\usepackage[a4paper,left=3cm,right=3cm, top=3cm, bottom=3cm]{geometry}
\usepackage[latin1]{inputenc}
\usepackage{mathrsfs}
\usepackage{graphicx}
\usepackage{epstopdf}
\usepackage{amsmath}
\usepackage{amsthm}
\usepackage{amssymb}
\usepackage{hyperref}
\usepackage{stmaryrd}
\usepackage{float}
\usepackage{bigints}
\usepackage{cite}
\usepackage{color}
\usepackage[abs]{overpic}
\usepackage[font=footnotesize,labelfont=bf]{caption}
\usepackage{cases}
\usepackage{tikz}
\usepackage{algorithm,algpseudocode}
\usepackage{rotating}
\usepackage{blkarray}
\usetikzlibrary{matrix,calc,arrows}
\usepackage[author={Lorenzo}]{pdfcomment}
\usepackage{soul,xcolor}
\usepackage{verbatim}
\usepackage{graphicx}
\usepackage{subcaption}
\usepackage{algorithm}
\usepackage{pifont}
\usepackage{multirow}

\restylefloat{table}
\theoremstyle{plain}
\newtheorem{thm}{Theorem}[section]

\newtheorem{lem}[thm]{Lemma}
\newtheorem{prop}[thm]{Proposition}
\theoremstyle{definition}

\theoremstyle{remark}
\newtheorem{remark}{Remark}

\newcommand{\eremk}{\hbox{}\hfill\rule{0.8ex}{0.8ex}}

\newcommand{\curl}{\text{rot}}
\newcommand{\rot}{\operatorname{rot}}
\newcommand{\curlbold}{\textbf{curl}}

\newcommand{\Hdiv}{H(\div,\Omega)}
\newcommand{\HdivE}{H(\div,\E)}
\newcommand{\HrotE}{H(\rot,\E)}
\newcommand{\HcurlE}{\HrotE}
\renewcommand{\div}{\operatorname*{div}}
\newcommand{\f}{f}
\newcommand{\n}{\mathbf n}
\newcommand{\nE}{\n_\E}
\renewcommand{\a}{a}
\newcommand{\an}{a_n}
\newcommand{\anomeganu}{a_n^{\omega_\nu}}
\newcommand{\aomeganu}{a^{\omega_\nu}}
\newcommand{\aE}{\a^\E}
\newcommand{\anE}{\aE_n}
\newcommand{\atilde}{\widetilde \a}
\newcommand{\atilden}{\widetilde \a_n}
\newcommand{\atildeE}{\atilde^\E}
\newcommand{\atildenE}{\atilden^\E}
\renewcommand{\b}{b}

\newcommand{\Sigmabold}{\boldsymbol \Sigma}
\newcommand{\SigmaboldgN}{\Sigmabold_{\gN}}
\newcommand{\Sigmaboldz}{\Sigmabold_{0}}
\newcommand{\Sigmaboldn}{\Sigmabold_n}
\newcommand{\SigmaboldngN}{\Sigmabold_{n,\gN}}
\newcommand{\Sigmaboldnz}{\Sigmabold_{n,0}}
\newcommand{\SigmaboldnE}{\Sigmaboldn(\E)}
\newcommand{\Vcaln}{\mathcal V_n}
\newcommand{\V}{V}
\newcommand{\Vtilde}{\widetilde V}
\newcommand{\VtildegD}{\Vtilde_{\gD}}
\newcommand{\Vtildez}{\Vtilde_{0}}
\newcommand{\Vn}{\V_n}
\newcommand{\Vnstar}{\Vn^*}
\newcommand{\VnE}{\Vn(\E)}
\newcommand{\Vtilden}{\Vtilde_n}
\newcommand{\VtildengD}{\Vtilde_{n,\gD}}
\newcommand{\Vtildenz}{\Vtilde_{n,0}}
\newcommand{\VtildenE}{\Vtilden(\E)}
\newcommand{\sigmabold}{\boldsymbol \sigma}
\newcommand{\taubold}{\boldsymbol \tau}
\newcommand{\taun}{\mathcal T_n}
\newcommand{\E}{K}
\newcommand{\e}{e}
\newcommand{\vtilden}{\widetilde v_n}
\newcommand{\utilden}{\widetilde u_n}
\newcommand{\vn}{v_n}
\newcommand{\un}{u_n}
\newcommand{\sigmaboldn}{\sigmabold_n}
\newcommand{\sigmaboldnnu}{\sigmaboldn^\nu}

\newcommand{\tauboldn}{\taubold_n}
\newcommand{\tauboldnnu}{\tauboldn^{\nu}}
\newcommand{\rnu}{r^{\nu}}
\newcommand{\rnnu}{r_n^{\nu}}
\newcommand{\rn}{r_n}
\newcommand{\omeganu}{\omega_\nu}
\newcommand{\phitildenu}{\widetilde \varphi_\nu}
\newcommand{\qn}{q_n}
\newcommand{\qnnu}{q_n^{\nu}}

\newcommand{\tauboldtilden}{\widetilde{\taubold}_n}
\newcommand{\tauboldboundaryn}{\taubold_n^{\partial \E}}
\newcommand{\p}{p}
\newcommand{\pE}{\p_\E}
\newcommand{\EE}{\mathcal E ^\E}
\newcommand{\boldalpha}{\boldsymbol \alpha}
\newcommand{\malphaE}{m_{\boldalpha}^\E}
\newcommand{\malphaboldE}{\mathbf m_{\alpha}^\E}

\newcommand{\Pinablap}{\widetilde\Pi^\nabla_\p}

\newcommand{\Piboldzp}{\boldsymbol\Pi^0_\p}

\newcommand{\qp}{q_\p}
\newcommand{\qppo}{q_{\p+1}}
\newcommand{\qboldp}{\mathbf q_\p}

\newcommand{\qpmo}{q_{\p-1}}
\newcommand{\qpmt}{q_{\p-2}}
\newcommand{\SE}{S^\E}
\newcommand{\StildeE}{\widetilde S^\E}

\newcommand{\alphatilde}{\widetilde \alpha}
\newcommand{\h}{h}
\newcommand{\hE}{\h_\E}
\newcommand{\he}{\h_\e}
\newcommand{\Pitildez}{\widetilde \Pi^0_{\p-2}}
\newcommand{\Gcal}{\boldsymbol{\mathcal G}}

\newcommand{\Pizpmo}{\Pi_{\p-1}^0}

\renewcommand{\k}{\kappa}
\newcommand{\xbold}{\mathbf x}
\newcommand{\xE}{\xbold_\E}
\newcommand{\Ibold}{\mathbf I}
\newcommand{\GammaD}{\Gamma_D}
\newcommand{\GammaN}{\Gamma_N}
\newcommand{\gD}{g_D}
\newcommand{\gN}{g_N}
\newcommand{\utilde}{\widetilde u}
\newcommand{\vtilde}{\widetilde v}

\newcommand{\varphitilde}{\widetilde \varphi}
\newcommand{\varphibold}{\boldsymbol \varphi}
\newcommand{\En}{\mathcal {E}_{n}}
\newcommand{\Enomeganu}{\En^{\omeganu}}
\newcommand{\EnI}{\En^I}
\newcommand{\EnB}{\En^B}
\newcommand{\Nun}{\mathcal {V}_{n}}
\newcommand{\NunI}{\Nun^I}
\newcommand{\NunB}{\Nun^B}
\newcommand{\NuE}{\mathcal V^\E}
\newcommand{\ctaubold}{c_{\taubold}}
\newcommand{\ctaubolde}{\ctaubold^\e}
\newcommand{\ctauboldGammaD}{\ctaubold^{\GammaD}}
\newcommand{\cB}{c_B}
\newcommand{\cOmega}{c_\Omega}
\newcommand{\phitilde}{\widetilde \varphi}

\newcommand{\cE}{c^{\E}}

\newcommand{\NE}{N^\E}
\newcommand{\etahyp}{\eta_{\mathfrak{eq}}}
\newcommand{\etahypE}{\eta_{\mathfrak{eq}, \E}}
\newcommand{\etaloc}{\eta_{\mathfrak{flux}}}
\newcommand{\etalocE}{\eta_{\mathfrak{flux}, \E}}

\newcommand{\taunomeganu}{\taun^{\omeganu}}
\newcommand{\gammatilde}{\widetilde \gamma}
\newcommand{\pbold}{\mathbf \p}
\newcommand{\NumEl}{N_{\taun}}

\newcommand{\etaRes}{\eta_{\mathfrak{res}}}
\newcommand{\etaResE}{\eta_{\mathfrak{res},\E}}
\newcommand{\Ihyp}{I_{\mathfrak{eq}}}
\newcommand{\Ires}{I_{\mathfrak{res}}}

\newcommand{\etahypbar}{\overline{\etahyp}}
\newcommand{\betan}{\beta_n}

\newcommand{\NsigmaboldnE}{N^{\Sigmaboldn}}
\newcommand{\NVtildeE}{N^{\Vtilden}}
\newcommand{\RT}{\mathbb{RT}_n}
\newcommand{\Qn}{Q_n}

\begin{tiny}
\author{
\normalsize{
}}
\end{tiny}

\date{}

\title{{\textbf{\large{Adaptive virtual element methods with equilibrated fluxes}}}}
\date{}
\author{{F. Dassi \thanks{Dip. di Matematica e Applicazioni,  Universit\`a degli Studi di Milano-Bicocca, Italy (franco.dassi@unimib.it, lorenzo.mascotto@unimib.it)}\quad
J. Gedicke \thanks{Institut f\"ur Numerische Simulation, Universit\"at Bonn, 53115 Bonn (gedicke@ins.uni-bonn.de)} \quad
L. Mascotto\footnotemark[1]\; \thanks{Fakult\"at f\"ur Mathematik, Universit\"at Wien, 1090 Vienna, Austria (lorenzo.mascotto@univie.ac.at)}}}

\begin{document}
\maketitle
\begin{abstract}
\noindent
We present an $\h\p$-adaptive virtual element method (VEM) based on the hypercircle method of Prager and Synge for the approximation of solutions to diffusion problems.
We introduce a reliable and efficient a posteriori error estimator, which is computed by solving an auxiliary global mixed problem.
We show that the mixed VEM satisfies a discrete inf-sup condition with inf-sup constant independent of the discretization parameters.
Furthermore, we construct a stabilization for the mixed VEM with explicit bounds in terms of the local degree of accuracy of the method.
The theoretical results are supported by several numerical experiments, including a comparison with the residual a posteriori error estimator.
The numerics exhibit the $\p$-robustness of the proposed error estimator.
In addition, we provide a first step towards the localized flux reconstruction in the virtual element framework,
which leads to an additional reliable a posteriori error estimator that is computed by solving local (cheap-to-solve and parallelizable) mixed problems.
We provide theoretical and numerical evidence that the proposed local error estimator suffers from a lack of efficiency.
\medskip


\end{abstract}

\section{Introduction}  \label{section:introduction}
Polygonal/polyhedral methods have several advantages over more standard technologies based on triangular/tetrahedral meshes.
For instance, when refining a mesh adaptively, the use of general shaped elements allows for the presence of hanging nodes and interfaces.
This simplifies the construction of hierarchies of meshes.
Amongst the various polytopal methods, the virtual element method (VEM) has received an increasing attention over the last years; see~\cite{VEMvolley}.

Adaptivity in the VEM has been applied to several problems: general elliptic problems in primal~\cite{cangianigeorgulispryersutton_VEMaposteriori, berrone2017residual, ManziniBeirao_VEMresidualaposteriori}
and mixed formulations~\cite{mixedVEMapos}, Steklov eigenvalue problems~\cite{MoraRiveraRodriguez_aposSteklov}, the elasticity equations~\cite{mora2017apost}, the $\h\p$-version of the VEM~\cite{hpVEMapos},
recovery-based VEM~\cite{chi2019simple}, superconvergent gradient recovery~\cite{superconvergentVEM_M3AS}, discrete fracture network flow simulations~\cite{berrone2019reliable},
parabolic problems with moving meshes~\cite{cangiani2020adaptive},
and anisotropic discretizations~\cite{antonietti2020anisotropic, weisser2019anisotropic}.
In all these references, residual error estimators were analyzed,
whereas equilibrated error estimators have not been investigated so far.

The hypercircle method for the computation of equilibrated error estimators was introduced by Prager and Synge in~\cite{prager1947approximations}; see also~\cite{aubin1971some}.
The idea behind it consists in constructing an error estimator based on the approximations of the primal \emph{and} the mixed formulations of the problem.
In the finite element framework, a combined error of the primal and mixed formulations is equal to a term involving the gradient of the solution to the primal method \emph{and} the equilibrated flux solution to the mixed method,
up to oscillation terms.

Braess and Sch\"oberl~\cite{braess2008equilibrated_curl}, and later Ern and Vohral{\'\i}k~\cite{ern2015polynomial} provided a major improvement to the hypercircle method.
They designed an error estimator based on equilibrated fluxes
using the solution to the primal method and a combination of discrete solutions to local (cheap-to-solve and parallelizable) mixed problems.

In the seminal works~\cite{braess2009equilibrated, braess2008equilibrated_curl}, Braess, Sch\"oberl, and collaborators proved that the equilibrated error estimator is reliable and has an efficiency that is independent of the polynomial degree~$\p$.
As discussed by Melenk and Wohlmuth~\cite[Theorem 3.6]{MelenkWohlmuth_hpFEMaposteriori}, this is not the case for the residual error estimator.
The $\p$-robustness of the equilibrated error estimator is proven in the discontinuous Galerkin (dG) setting as well; see, e.g., \cite{ern2015polynomial} and the references therein.
Hence, they are very well suited for $\h\p$-adaptivity; see, e.g., \cite{dolejsi2016hp}.

Local flux reconstruction techniques have been applied to various problems, e.g., parabolic problems~\cite{ErnSmearsVohralik2017parabolic},
reaction diffusion problems~\cite{SmearsVohralik_singularly}, the Helmholtz equation~\cite{ChFreErnVohralik_Helmholtz_equilibrated,congreve2019robust} and magnetostatic problems~\cite{equilibrated_magnetostatic}.
\medskip

This paper represents the first attempt to combine the hypercircle method with the VEM.
In particular, we want to dovetail the geometric flexibility of the VEM with the robustness properties of the hypercircle method and the flux reconstruction.
We analyze the hypercircle method for the~$\h$-, $\p$-, and~$\h\p$-versions of the VEM.

The structure and contents of the paper are as follows.
In Section~\ref{section:VEM}, we introduce the VEM for the primal and mixed formulations of a two dimensional diffusion problem.
Despite the construction of the two methods is well understood~\cite{VEMvolley, Brezzi-Falk-Marini, BBMR_generalsecondorder}, there are two issues that we want to address, which have not been covered in the literature so far.
We show that 
\begin{itemize}
\item the mixed formulation of the VEM satisfies a discrete inf-sup condition with inf-sup constant independent of the degree of accuracy of the method;
\item we construct a stabilization of the mixed VEM with stability bounds that are explicit in terms of such degree of accuracy.
\end{itemize}
In Section~\ref{section:apos}, we introduce an equilibrated error estimator and prove its reliability and efficiency.
Such an error estimator consists of two terms. One is similar to the FEM equilibrated error estimator; the other involves two stabilization terms, typical of the VEM framework.
Due to the several variational crimes of the VEM, we observe a loss of the $\p$-robustness due to the presence of the stabilization terms,
as well as the loss of the ``constant free'' nature of the reliability bound.
Numerical experiments are presented in Section~\ref{section:nr}. Amongst them, we show that the proposed error estimator is $\p$-robust, differently from the residual error estimator of~\cite{hpVEMapos}.
Moreover, we display the performance of the $\h$- and the $\h\p$-hypercircle method based on the Melenk-Wohlmuth's $\h\p$-refining strategy presented in~\cite{MelenkWohlmuth_hpFEMaposteriori}.

In Section~\ref{section:localization-mixed}, we present a first step towards the analysis of the local flux reconstruction in the VEM framework:
we discuss how to design a reliable error estimator using local VE flux reconstructions.
Amongst the various technical tools needed in the analysis, we provide
\begin{itemize}
\item the design of a high-order virtual element partition of unity, which differs from the standard one introduced in~\cite{XVEM_2019, Helmholtz-VEM};
\item the design of local mixed VEMs, satisfying an equilibration condition on fluxes and a residual-type equation.
\end{itemize}
Numerical results with this new error estimator are the topic of Section~\ref{section:nr:local}.
Here, we check that the equilibration condition is fulfilled, which guarantees reliability. However, notably for high-order methods, we have theoretical and numerical evidence that the local error estimator suffers from a lack of efficiency.
This paves the way to other approaches where a deeper analysis on the design of local mixed problems has to be performed.
We draw some conclusions in Section~\ref{section:conclusion}.

\paragraph*{Notation.}
Throughout the paper, we employ a standard notation for Sobolev spaces. Given a measurable open set~$D \subset \mathbb R^2$ and~$s \in \mathbb N$,
$L^2(D)$ and~$H^s(D)$ denote the standard Lebesgue and Sobolev spaces endowed with inner products~$(\cdot, \cdot)_{s,D}$ and seminorms~$\vert \cdot \vert_{s,D}$.
We set the Sobolev norm of order~$s$ as
\[
\Vert \cdot \Vert _{s,D} : =
\Vert \cdot \Vert_{0, D}
+ \sum_{\ell=1}^s \vert \cdot \vert_{\ell, D}.
\]
The case~$s=1$ is special and, when no confusion occurs, we shall write
\[
a^D(\cdot, \cdot) = (\cdot, \cdot)_{1,D}
:= (\nabla \cdot,\nabla\cdot)_{0,D}.
\]
We define fractional order Sobolev spaces via interpolation theory~\cite{Triebel},
whereas we define Sobolev negative order spaces by duality as
\[
H^{-1}(D) := [H^1_0(D)]^*, \quad \quad H^{-\frac{1}{2}} (\partial D) := [H^{\frac{1}{2}}(\partial D)]^*,
\]
and endow them with the norms
\begin{equation} \label{negative:norms}
\Vert v \Vert_{-1,D} := \sup_{w \in H^1_0(D),\, w \ne 0} \frac{(v,w)_{0,D}}{\vert w \vert_{1,D}},\quad \quad \Vert v \Vert_{-\frac{1}{2}, \partial D} := \sup_{w \in H^{\frac{1}{2}}(\partial D),\, w \ne 0 } \frac{(v,w)_{0,\partial D}}{\Vert w \Vert _{\frac{1}{2}, \partial D}}.
\end{equation}
Recall the definition of the differential operators
\[
\div = \partial _x +\partial_y, \quad \quad \quad \rot = \partial_y - \partial_x ,
\]
and introduce the~$H(\div)$ and~$H(\curl)$ spaces
\[H(\div,D) := \left \{ \taubold \in [L^2(D)]^2  \mid \div \taubold \in L^2(D)  \right \}, \quad H(\rot,D) := \left \{ \taubold \in [L^2(D)]^2  \mid \rot \taubold \in L^2(D)    \right \}.
\]
For all~$\ell \in \mathbb N$, $\mathbb P_\ell(D)$ denotes the space of polynomials of degree at most~$\ell$ over~$D$.
We shall employ the multi-indices~$\boldalpha \in \mathbb N^2$ to describe the basis elements of~$\mathbb P_\ell (D)$. To this purpose, we shall use the natural bijection~$\mathbb N \leftrightarrow \mathbb N_0^2$ given by
\begin{equation} \label{bijection:multi}
1\leftrightarrow (0,0), \quad 2\leftrightarrow (1,0), \quad 3 \leftrightarrow (0,1),\quad 4 \leftrightarrow (2,0), \quad 5 \leftrightarrow (1,1), \quad 6 \leftrightarrow (0,2), \quad \dots
\end{equation}
As a matter of style, we employ the following notation.
Given two positive quantities~$a$ and~$b$, we write~$a \lesssim b$ if there exists a positive constant~$C$ such that~$a \le C \, b$.
We write~$a \approx b$ if~$a \lesssim b$ and~$b \lesssim a$ are valid at the same time.

\paragraph*{The model problem.}
Let~$\Omega \subset \mathbb R^2$ be a polygonal domain with boundary~$\partial \Omega$ split into~$\partial\Omega = \GammaD \cup \GammaN$ with~$\GammaD \ne \emptyset$ and~$\GammaD \cap \GammaN = \emptyset$.
Denote the outward unit normal vector on~$\GammaN$ by~$\n$. Let~$\k$ be a smooth scalar function such that there exist two positive constants~$k_* < k^*$ satisfying
\begin{equation} \label{ass:k}
0 < k_*  \le  \k(\xbold)  \le k^* < +\infty \quad\quad \text{for almost all } \xbold \in \Omega.
\end{equation}
\emph{The primal formulation.}
Let $\f \in L^2(\Omega)$, $\gD\in H^{\frac{1}{2}}(\GammaD)$, and~$\gN\in H^{-\frac{1}{2}}(\GammaN)$.
We aim to approximate the solution to the problem: find~$\utilde$ such that
\begin{equation} \label{strong:primal}
\begin{cases}
-\div (\k\, \nabla \utilde) = \f & \text{in } \Omega\\
\utilde = \gD 			& \text{on } \GammaD \\
\n \cdot (\k \nabla \utilde) = \gN	& \text{on } \GammaN. \\
\end{cases}
\end{equation}
Define the spaces
\[
\VtildegD := \{ \vtilde \in H^1(\Omega)\mid \vtilde = \gD \text{ on } \GammaD   \} ,	\quad\quad  \Vtildez:= \{ \vtilde \in H^1(\Omega) \mid \vtilde = 0 \text{ on } \GammaD   \} ,
\]
and the bilinear form
\begin{equation} \label{notation:strong}
\atilde (\utilde, \vtilde) := (\k\nabla \utilde, \nabla \vtilde)_{0,\Omega} \quad\forall \utilde,\,  \vtilde  \in H^1(\Omega). 
\end{equation}
The weak formulation of problem~\eqref{strong:primal} reads
\begin{equation} \label{primal:formulation}
\begin{cases}
\text{find } \utilde \in \VtildegD \text{ such that}\\
\atilde(\utilde, \vtilde) = (\f, \vtilde)_{0,\Omega} + (\gN, \vtilde)_{0,\GammaN}   \quad \forall v\in \Vtildez.
\end{cases}
\end{equation}
The term~$(\gN, \vtilde)_{0,\GammaN}$ has to be understood as a duality pairing between~$H^{-\frac{1}{2}}(\GammaN)$ and~$H^{\frac{1}{2}}(\GammaN)$.
\medskip

\noindent
\emph{The mixed formulation.}
Define the spaces
\[
\begin{split}
&\SigmaboldgN := \{ \taubold \in \Hdiv \mid \n \cdot \taubold = \gN \text{ on } \GammaN \}, \\
& \Sigmaboldz := \{ \taubold \in \Hdiv \mid \n \cdot \taubold = 0 \text{ on } \GammaN \} ,\quad\quad \V: = L^2(\Omega) ,\\
\end{split}
\]
and the bilinear forms
\begin{equation} \label{bilinear:b}
\begin{split}
& \a(\sigmabold, \taubold) := \int_\Omega \k^{-1} \sigmabold \cdot \taubold \quad \forall \sigmabold,\, \taubold \in \Hdiv, \\
&  \b(\taubold, v) := -\int_\Omega \div(\taubold) v\quad \forall \taubold \in \Hdiv,\, \forall v \in \V.
\end{split}
\end{equation}
We point out that~$\VtildegD$ and~$\SigmaboldgN$ are not vector spaces in general.

Consider the mixed formulation of problem~\eqref{primal:formulation}. In strong formulation, it consists in finding~$\sigmabold$ and~$u$ such that
\[
\begin{cases}
\k^{-1}\sigmabold=-\nabla u 	& \text{in } \Omega\\
\div \sigmabold = \f 			& \text{in } \Omega\\
u = \gD 					& \text{on } \GammaD\\
\n\cdot \sigmabold = -\gN			& \text{on } \GammaN,\\
\end{cases}
\]
whereas, in weak formulation, it reads
\begin{equation} \label{mixed:formulation}
\begin{cases}
\text{find } (\sigmabold, u) \in \SigmaboldgN \times \V \text{ such that}\\ 
\a(\sigmabold, \taubold)  +\b(\taubold, u) = -(\gD, \n \cdot \taubold)_{0,\GammaD} \quad \forall \taubold\in \Sigmaboldz \\
\b(\sigmabold, v) = (-\f,v)_{0,\Omega} \quad \forall v\in \V.
\end{cases}
\end{equation}
The term~$(\gD,  \n \cdot \taubold)_{0,\GammaD}$ has to be understood as a duality pairing between~$H^{\frac{1}{2}}(\GammaD)$ and~$H^{-\frac{1}{2}}(\GammaD)$.

The well-posedness of~\eqref{primal:formulation} and~\eqref{mixed:formulation} is a consequence of a lifting argument, and the Lax-Milgram lemma and the standard inf-sup theory, respectively; see, e.g., \cite{BrezziFortin}.

\begin{remark}
Formulations~\eqref{primal:formulation} and~\eqref{mixed:formulation} are equivalent.
Moreover, given their solutions~$\utilde$ and~$\sigmabold$, the following identity is valid:
\begin{equation} \label{strong:identity}
\nabla \widetilde u = -\k^{-1} \sigmabold.
\end{equation}
\end{remark}

\begin{remark} \label{remark:tilde}
As a matter of style, we shall employ the following notation in the remainder of the paper.
We shall use a~$\sim$ whenever referring to functions, spaces, etc. associated with the primal formulation.
No~$\sim$ is employed as for the mixed formulation.
\end{remark}

\section{Virtual element discretization} \label{section:VEM}
In this section, we introduce the VEM for both the primal~\eqref{primal:formulation} and the mixed~\eqref{mixed:formulation} formulations.
More precisely, we introduce the notation and certain assumptions for the polygonal meshes and the data of the problems in Section~\ref{subsection:meshes}.
Next, we describe the virtual element methods for the discretization of the primal~\eqref{primal:formulation} and mixed~\eqref{mixed:formulation} formulations in Sections~\ref{section:VEM:primal} and~\ref{section:VEM:mixed}, respectively.
Section~\ref{subsection:stabilizations} deals with the construction of explicit stabilizations, whereas we prove the well-posedness of the two methods in Section~\ref{subsection:well}.

\subsection{Polygonal meshes} \label{subsection:meshes}
Consider~$\{\taun\}_{n \in \mathbb N}$ a sequence of decompositions of~$\Omega$ into polygons with straight edges.
Hanging nodes are dealt with as standard nodes.
Given a mesh~$\taun$, we denote its set of vertices, boundary vertices, and internal vertices by~$\Nun$, $\NunB$, and~$\NunI$, respectively.
Furthermore, we denote its set of edges, boundary edges, and internal edges by~$\En$, $\EnB$, and~$\EnI$, respectively.

With each~$\E \in \taun$, we associate~$\NuE$ its set of vertices and~$\EE$ its set of edges.
We denote  its diameter and its centroid by~$\hE$ and~$\xE$,
whereas~$\nE$ denotes the outward pointing, normal vector of~$\E$.
Finally, for all edges~$\e \in \En$, we fix once and for all~$\n_\e$, the unit normal vector associated with~$\e$ and denote the length of~$\e$ by~$\he$.

For all~$n \in \mathbb N$, $\taun$ is a conforming polygonal decomposition, i.e., every internal edge belongs to the intersection of the boundary of two neighbouring elements.
Also, $\taun$ is conforming with respect to the Dirichlet and Neumann boundaries.
In other words, for all boundary edges~$\e \in \EnB$, either~$\e \subset \GammaD$ or~$\e \subset \GammaN$.

We assume the following properties on the meshes and the data of problems~\eqref{primal:formulation} and~\eqref{mixed:formulation}: for all~$n \in \mathbb N$,
\begin{itemize}
\item[(\textbf{G1})] every~$\E \in \taun$ is star-shaped with respect to a ball of radius greater than or equal to~$\gamma \hE$, for a positive constant~$\gamma$;
\item[(\textbf{G2})] given~$\E \in \taun$, for all its edges~$\e \in \EE$, $\hE$ is smaller than or equal to~$\gammatilde \he$, for a positive constant~$\gammatilde$;
\item[(\textbf{K})] the diffusion parameter~$\k$ is piecewise constant over $\{\taun\}_n$;
\item[(\textbf{D})] the boundary data~$\gD$ and~$\gN$ are piecewise polynomials of a given degree~$\p \in \mathbb N$ over~$\GammaD$ and~$\GammaN$.
\end{itemize}
We employ the assumptions~(\textbf{G1}) and~(\textbf{G2}) in the analysis of the forthcoming sections.
Instead, we call for the assumptions~(\textbf{K}) and~(\textbf{D}) to simplify the analysis.
When no confusion occurs, we denote the piecewise divergence operator over~$\taun$ by~$\div$.

Assumptions (\textbf{G1})--(\textbf{G2}) can be generalized \cite{cao2018anisotropic, brennerVEMsmall, beiraolovadinarusso_stabilityVEM} to the case of small edges and anisotropic elements.
We stick to the current setting for the sake of simplicity.
Further, to prove explicit stability bounds, see Proposition~\ref{proposition:stab:mixed} below, we shall also demand the convexity of the elements.

\subsection{Virtual elements for the primal formulation} \label{section:VEM:primal}
Here, we introduce the virtual element discretization of problem~\eqref{primal:formulation}, mimicking what is done in~\cite{BBMR_generalsecondorder}.
We assume that degree of accuracy~$\p \in \mathbb N$ is uniform over all the elements; see Remark~\ref{remark:hp:primal} below for the variable degree case.

\paragraph*{Virtual element spaces.}
Given an element~$\E \in \taun$, we define the local nodal virtual element space on~$\E$ as
\[
\VtildenE :=\{ \vtilden \in C^0(\overline \E) \mid \vtilden{}_{|\e} \in \mathbb P_\p(\e) \text{ for all } \e\in \EE, \quad \Delta \vtilden \in \mathbb P_{\p-2}(\E)  \}.
\]
We observe that~$\mathbb P_\p(\E) \subseteq \Vtilden$.
Further, the functions in~$\Vtilden$ are available in closed form on~$\partial \E$ but not in~$\E$.

Given the multi-indices~$\boldalpha$ as in bijection~\eqref{bijection:multi}, let~$\{\malphaE\}_{\vert \boldalpha \vert = 0} ^{\p-2}$ be any basis of~$\mathbb P_{\p-2}(\E)$.
We assume that the basis elements~$\malphaE$ are invariant with respect to dilations and translations\footnote{Here and in what follows, invariance with respect to dilations and translations means that we consider polynomials that are shifted and scaled with respect to the barycenter and diameter of the element.}.
For all~$\vtilden \in \VtildenE$, consider the following set of linear functionals:
\begin{itemize}
\item the point values of~$\vtilden$ at the vertices of~$\E$;
\item the point values of~$\vtilden$ at the~$\p-1$ internal Gauss-Lobatto nodes of each  edge of~$\E$;
\item (scaled) moments
\begin{equation} \label{internal:moments:primal}
\frac{1}{\vert \E \vert} \int_\E \vtilden\, \malphaE \quad \quad \forall \vert \boldalpha \vert =0,\dots,\p-2.
\end{equation}
We are going to discuss possible choices of the polynomial basis in Sections~\ref{section:nr} and~\ref{section:nr:local}.
\end{itemize}
\begin{prop} 
For all~$\E \in \taun$, the above set of linear functionals is a set of unisolvent degrees of freedom for~$\VtildenE$.
\end{prop}
\begin{proof}
See~\cite[Section~4.1]{BBMR_generalsecondorder}.
\end{proof}
The global virtual element space with no boundary conditions is obtained by merging the local spaces continuously:
\[
\Vtilden := \left \{ \vtilden \in \mathcal C^0(\overline\Omega) \mid \vtilden{}_{|\E}  \in \VtildenE \; \forall \E \in \taun \right\}.
\]
We incorporate the Dirichlet boundary conditions in the space by imposing the degrees of freedom associated with the edges in~$\GammaD$.
We define the discrete trial and test spaces
\[
\begin{split}
\VtildengD 	& := \left\{ \vtilden \in \Vtilden \mid \vn= \gD \text{ on }\e \text{ if } \e \subset \GammaD    \right\}, \\
\Vtildenz    	& := \left\{ \vtilden \in \Vtilden \mid \vn = 0 \text{ on }\e \text{ if } \e \subset \GammaD   \right\}. \\
\end{split}
\]
We associate with each global space a set of unisolvent degrees of freedom, obtained by an $H^1$ conforming coupling of their local counterparts.

\paragraph*{Projectors.}
For all~$\E \in \taun$, by means of the degrees of freedom, we can compute the local $H^1$ projector $\Pinablap : \VtildenE \rightarrow \mathbb P_\p(\E)$ defined as
\begin{equation} \label{projection:nabla}
\atildeE(\vtilden - \Pinablap \vtilden, \qp) = 0 \quad \forall \qp \in \mathbb P_\p(\E),\quad \quad \int_{\partial \E} (\vtilden - \Pinablap\vtilden) =0 \quad \forall \vtilden \in \Vtilden.
\end{equation}
Furthermore, for all~$\E \in \taun$ and~$\p\ge 2$, we introduce the local~$L^2(\E)$ projector $\Pitildez : \Vtilden(\E) \rightarrow \mathbb P_{\p-2}(\E)$ defined as
\[
(\vtilden - \Pitildez \vtilden, \qpmt)_{0,\E} = 0\quad \quad \forall \vtilden\in \VtildenE,\, \forall \qpmt \in \mathbb P_{\p-2}(\E).
\]
This projector is computable via the bubble degrees of freedom~\eqref{internal:moments:primal}. 

\paragraph*{Discrete bilinear forms and right-hand side.}
The functions in~$\Vtilden$ are known on the skeleton of the mesh only.
Consequently, for all~$\utilden$ and~$\vtilden$ in~$\Vtilden$, it is not possible to compute the bilinear form~$\atilde(\utilden,\vtilden)$~\eqref{notation:strong} explicitly.
Thence, following the VEM gospel~\cite{BBMR_generalsecondorder}, we split the global bilinear form~$\atilde(\cdot, \cdot)$ into a sum of local contributions:
\[
\atilde(\utilde, \vtilde) = \sum_{\E \in \taun} \int_\E \k  \nabla \utilde \cdot \nabla \vtilde =: \sum_{\E \in \taun} \atildeE (\utilde_{|\E}, \vtilde_{|\E})  \quad\quad \forall \utilde, \, \vtilde \in \Vtilde.
\]
We allow for the following variational crime in the design of the local bilinear form.
Let~$\StildeE : \ker(\Pinablap) \times \ker(\Pinablap) \rightarrow \mathbb R$ be any bilinear form computable via the degrees of freedom and satisfying
\begin{equation} \label{local_stab:primal}
\alphatilde_* \atildeE( \vtilden, \vtilden) \le \StildeE(\vn,\vn) \le \alphatilde^* \atildeE( \vtilden, \vtilden) \quad\quad \forall \E \in \taun, \, \forall \vtilden \in \ker (\Pinablap).
\end{equation}
The constants~$0< \alphatilde_* \le \alphatilde^*< +\infty$ depend possibly on the geometry of the polygonal decomposition through the parameter~$\gamma$ in the assumptions (\textbf{G1}) and (\textbf{G2}), the degree of accuracy~$\p$, and~$\k$, but must be independent of~$\hE$.

Introducing the local discrete bilinear forms
\[
\atildenE(\utilden, \vtilden) := (\k \nabla \Pinablap \utilden, \nabla \Pinablap  \vtilden) + \StildeE( (I-\Pinablap) \utilden, (I-\Pinablap) \vtilden )  \quad\quad \forall \utilden,\, \vtilden\in \VtildenE,
\]
we define the global discrete bilinear form
\[
\atilden(\utilden, \vtilden) := \sum_{\E \in \taun} \atildenE(\utilden{}_{|\E}, \vtilden{}_{|\E})	 		\quad\quad \forall \utilden,\, \vtilden \in \Vtilden.
\]
As discussed in~\cite{VEMvolley}, the bilinear form~$\atilden(\cdot, \cdot)$ is coercive and continuous with constants $\min(k_*, \min_{\E \in \taun} \alphatilde_*)$ and~$\max(k^*, \max_{\E \in \taun} \alphatilde^*)$.
Explicit choices of the stabilization in~\eqref{local_stab:primal} are detailed in Section~\ref{subsection:stabilizations}.

As for the treatment of the right-hand side in the case~$\p=1$, we refer the reader to~\cite{VEMvolley} for details,
whilst, for~$\p \ge 2$, we approximate the right-hand side~$(\f,\vtilden)_{0,\Omega}$ perpetrating the following variational crime:
\[
\langle \f, \vtilden \rangle_n : = \sum_{\E \in \taun} \langle \f, \vtilden{}_{|\E} \rangle_{n,\E} := \sum_{\E \in \taun} (\f, \Pitildez \vtilden{}_{|\E})_{0,\E} \quad\quad \forall \vtilden \in \Vtilden.
\]

\paragraph*{The virtual element method for the primal formulation.}
The virtual element method tailored for the approximation of the problem in primal formulation~\eqref{primal:formulation} reads
\begin{equation} \label{VEM:primal}
\begin{cases}
\text{find } \utilden \in \VtildengD \text{ such that}\\
\atilden(\utilden, v) = \langle \f, \vtilden\rangle_{n} + (\gN, \vtilden)_{0,\GammaN} \quad \forall \vtilden \in \Vtildenz.
\end{cases}
\end{equation}

\begin{remark}  \label{remark:hp:primal}
So far, we have discussed the construction of virtual elements with uniform degree of accuracy over all the elements.
Indeed, the flexibility of the virtual element framework allows for the construction of global spaces with variable degrees of accuracy.
Let~$\taun$ be a mesh with~$\NumEl$ elements.
Consider~$\pbold \in \mathbb N^{\NumEl}$ and associate with each~$\E \in \taun$ a degree of accuracy~$\pE$.
To each internal edge, we associate the maximum of the degrees of accuracy of the two neighbouring elements, whereas, to each boundary edge, we associate the degree of accuracy of the only neighbouring element.
The definition and cardinality of the bulk and edge degrees of freedom is performed accordingly.
We refer to~\cite[Section 3]{hpVEMcorner} for a thorough presentation of the variable degree virtual element spaces case.
\eremk
\end{remark}

\subsection{Virtual elements for the mixed formulation} \label{section:VEM:mixed}
In this section, we discuss the virtual element discretization of the mixed formulation~\eqref{mixed:formulation}; see also~\cite{Brezzi-Falk-Marini}.
For the case of piecewise analytic diffusivity tensor~$\k$, we refer the reader to~\cite{BBMR_generalsecondorder, HdivHcurlVEM}.
We consider a uniform degree of accuracy~$\p \in \mathbb N$ over all the elements; see Remark~\ref{remark:hp:mixed} below for the variable degree of accuracy case.

\paragraph*{Virtual element spaces.}
We define the virtual element spaces for both the primal and the flux variables.
As for the primal variable space, we define~$\Vn \subset \V$ as the space of piecewise discontinuous polynomials of degree~$\p-1$ over~$\taun$, i.e.,
\[
\Vn :=  \mathcal S^{\p-1, -1} (\Omega, \taun) .
\]
Given the multi-indices~$\boldalpha$ as in bijection~\eqref{bijection:multi},
let~$\{\malphaE \}_{\vert \boldsymbol \alpha \vert=0}^{\p-1}$ be any basis of~$\mathbb P_{\p-1}(\E)$.
We assume that the basis elements~$\malphaE$ are invariant with respect to dilations and translations.
A set of unisolvent degrees of freedom is provided by scaled moments: given~$\vn\in \VnE$,
\[
\frac{1}{\vert \E \vert} \int_\E \vn \malphaE \quad \quad \forall \vert \boldalpha \vert=0,\dots,\p-1.
\]
The construction of the flux spaces is as follows. On each element~$\E$, we define the local space
\begin{equation} \label{local:mixed:VES}
\begin{split}
\SigmaboldnE := \big\{ \tauboldn \in \HdivE  \cap \HcurlE \mid  	& \n \cdot \tauboldn{}_{|\e} \in \mathbb P_\p(\e) \text{ for all } \e\in \EE, \\
				&  \div \tauboldn \in \mathbb P_{\p-1}(\E), \; \rot \tauboldn \in \mathbb P_{\p-1}(\E)    \big\}.  \\
\end{split}
\end{equation}
Observe that~$[\mathbb P_{\p}(\E)]^2 \subseteq \SigmaboldnE$ and that the functions in~$\SigmaboldnE$ are neither known in closed form on~$\partial \E$ nor inside the element~$\E$.

For all~$\p \in \mathbb N$, define
\[
\Gcal_{\p}(\E) := \nabla \mathbb P_{\p + 1}(\E).
\]
For all $\E \in \taun$, let~$\{ \malphaboldE \}_{\alpha =1 }^{\dim(\Gcal_\p(\E))}$ be a basis of~$\Gcal_\p(\E)$.
We assume that the basis elements~$\malphaboldE$ are invariant with respect to dilations and translations.
An explicit construction of the basis for the polynomial space~$\Gcal_{\p}(\E)$ can be found, e.g., in~\cite[Proposition 2.1]{VEMbricksMixed}.
Further, introduce $\{ \malphaE  \}_{\vert \boldalpha \vert =0}^{\p}$, a basis of~$\mathbb P_{\p}(\E)$.
We assume that the basis elements are invariant with respect to dilations and translations.

Consider the following set of linear functionals: given~$\tauboldn \in \SigmaboldnE$,
\begin{itemize}
\item for all edges~$\e \in \EE$, the evaluation at the $\p+1$ Gauss nodes~$\{\nu_j^\e\}_{j=0}^\p$ of~$\e$
\begin{equation} \label{edgeDOF:flux}
(\n \cdot \tauboldn) (\nu_j) \quad\quad \forall j = 0,\dots,\p;
\end{equation}
\item the gradient-like moments
\begin{equation} \label{gradientDOF:flux}
\frac{1}{\vert \E \vert}  \int_\E \tauboldn \cdot \malphaboldE \quad\quad \forall \alpha = 1, \dots, \dim(\Gcal_{\p-2}(\E));
\end{equation}
\item the rotor-like moments
\begin{equation} \label{orthogonalDOF:flux}
\frac{\hE}{\vert \E \vert}  \int_\E \rot(\tauboldn) \malphaE \quad\quad \forall \vert \boldalpha \vert = 0,\dots , \p-1.
\end{equation}
\end{itemize}

\begin{prop} \label{prop:unisolvency}
The set of linear functionals in~\eqref{edgeDOF:flux}--\eqref{orthogonalDOF:flux} is a set of unisolvent degrees of freedom for~$\SigmaboldnE$, for all~$\E \in \taun$.
\end{prop}
\begin{proof}
Since the dimension of~$\SigmaboldnE$ equals the number of linear functionals~\eqref{edgeDOF:flux}--\eqref{orthogonalDOF:flux},
it is enough to show that given~$\tauboldn \in \Sigmaboldn$ with functionals equals to zero is zero. In particular, we show that~$\tauboldn$ solves a div-rot problem with zero data.

Clearly, we have~$\tauboldn{}_{|\partial \E} =0$, thanks to the boundary degrees of freedom~\eqref{edgeDOF:flux}.
Moreover, $\div \tauboldn = 0$. In fact, thanks to the definition of the degrees of freedom~\eqref{edgeDOF:flux}--\eqref{gradientDOF:flux}, an integration by parts yields
\begin{equation} \label{divergence:explicit}
\int_\E \div(\tauboldn) \qpmo = -\int_\E \tauboldn \cdot \nabla \qpmo + \int_{\partial \E} \taubold\cdot \n \, \qpmo = 0  \quad \quad \forall \qpmo \in \mathbb P_{\p-1}(\E).
\end{equation}
Eventually, thanks to rotor moments~\eqref{orthogonalDOF:flux}, we get~$\rot \tauboldn = 0$.
\end{proof}
Define the jump operator~$\llbracket \cdot \rrbracket _{\e}$ across an edge~$\e$ as follows.
If~$\e$ is an internal edge shared by the elements~$\E_1$ and~$\E_2$ with outward unit normal vectors~$\n_{\E_1}$ and~$\n_{\E_2}$, respectively, and given~$\n_\e$ the global unit normal vector associated with~$\e$,
set
\[
\llbracket \sigmaboldn \rrbracket _{\e} = (\n_\e\cdot \n_{\E_1})\n_\e \cdot \sigmabold_{\E_1} + (\n_\e\cdot \n_{\E_2})  \n_\e \cdot \sigmabold_{\E_2}.
\]
Instead, if~$\e$ is a boundary edge, set~$\llbracket \sigmaboldn \rrbracket_\e = \n_\e \cdot \sigmaboldn $.
For all~$\E\in \taun$ and~$\e \in \EE$, we have that $\n_\e\cdot\nE{} = \pm 1$ depending on the choice of~$\n_\e$.

We define the global space~$\Sigmaboldn$ without boundary conditions by coupling the normal components at the internal interfaces between elements sharing an edge:
\[
\Sigmaboldn := \left \{  \sigmaboldn \in H(\div, \Omega) \mid \sigmaboldn{}_{|\E} \in \Sigmaboldn(\E) \; \forall \E \in \taun, \quad \llbracket \sigmaboldn  \rrbracket _{\e} =0\; \forall \e \in  \EnI    \right \}.
\]
We incorporate the boundary condition~$\gN$ in the space by imposing the degrees of freedom associated with the edges on the Neumann part of the boundary~$\GammaN$.
We set the discrete trial and test spaces for the fluxes space as
\[
\begin{split}
\SigmaboldngN 	& := \left\{ \tauboldn \in \Sigmaboldn \mid \n  \cdot \tauboldn  = \gN \text{ on } \e \text{ if } \e\subset \GammaN    \right\}, \\
\Sigmaboldnz  	& := \left\{ \tauboldn \in \Sigmaboldn \mid \n \cdot \tauboldn  = 0 \text{ on } \e \text{ if } \e \subset \GammaN   \right\}. \\
\end{split}
\]
With each global space, we associate a set of unisolvent degrees of freedom obtained by coupling the boundary degrees of freedom.

\begin{remark} \label{remark:divergence-mixed}
Although the functions in~$\SigmaboldnE$ are not available in closed form inside the elements, their divergence is computable from the degrees of freedom~\eqref{edgeDOF:flux} and~\eqref{gradientDOF:flux} \emph{explicitly}; see~\eqref{divergence:explicit}.
Thence, we are able to compute exactly the bilinear form~$\b(\tauboldn, \vn)$~\eqref{bilinear:b} for all~$\tauboldn\in \Sigmaboldn$ and~$\vn \in \Vn$.
\eremk
\end{remark}

\paragraph*{Projector.}
We introduce  the~$L^2(\E)$ vector projector~$\Piboldzp : \SigmaboldnE \rightarrow \Gcal_{\p}(\E)$ into the gradient of polynomials on each element~$\E \in \taun$:
\begin{equation} \label{projector:mixed}
\aE( \tauboldn - \Piboldzp \tauboldn, \qboldp) = 0 \quad\quad \forall \tauboldn\in \SigmaboldnE,\, \forall \qboldp \in \Gcal_{\p}(\E).
\end{equation}
This projector is computable from the degrees of freedom in~\eqref{edgeDOF:flux}--\eqref{orthogonalDOF:flux}.
To see this, we observe that
\[
\aE(\tauboldn, \nabla \qppo) = -\aE(\div\tauboldn, \qppo)  + (\tauboldn \cdot \n, \qppo)_{0,\partial \E}  \quad \forall \tauboldn \in \SigmaboldnE, \, \forall \qppo \in \mathbb P_{\p+1}(\E).
\]
The terms on the right-hand side are computable using an integration by parts and the degrees of freedom~\eqref{edgeDOF:flux} and~\eqref{gradientDOF:flux}.
Indeed, both the divergence and the normal components over the mesh skeleton of functions in virtual element spaces are known explicitly; see also~\eqref{divergence:explicit}.

The projector~$\Piboldzp$ can be generalized to a projector into spaces of gradients of polynomials with arbitrary degree, i.e., $[\mathbb P_{\widetilde \p}(\E)]^2$ for all~$\widetilde\p \in \mathbb N$.

\paragraph*{Discrete bilinear forms and right-hand side.}
As for the bilinear form~$\a(\cdot,\cdot)$, we proceed similarly to what is done for the primal formulation in Section~\ref{section:VEM:primal}. First, consider the splitting
\[
\a(\sigmabold, \taubold) = \sum_{\E \in \taun} \int_\E \k^{-1} \sigmabold \cdot \taubold =: \sum_{\E \in \taun}  \aE (\sigmabold_{|\E}, \taubold_{|\E})  \quad\quad \forall \sigmabold,\, \taubold \in \Sigmabold.
\]
The projector~$\Piboldzp$ in~\eqref{projector:mixed} allows us to construct a computable discrete bilinear form mimicking~$a(\cdot, \cdot)$.
For all~$\E \in \taun$, consider any bilinear form~$\SE : \ker(\Piboldzp) \times \ker(\Piboldzp) \rightarrow \mathbb R$, computable via the degrees of freedom and satisfying
\begin{equation} \label{local_stab:mixed}
\alpha_* \aE (\tauboldn, \tauboldn) \le \SE(\tauboldn,\tauboldn) \le \alpha^* \aE (\tauboldn, \tauboldn) .
\end{equation}
The constants~$0< \alpha_* \le \alpha^*< +\infty$ depend possibly on the geometry of the polygonal decomposition through the parameter~$\gamma$ in the assumptions (\textbf{G1}) and (\textbf{G2}), the degree of accuracy~$\p$, and~$\k$, but must be independent of the size~$\hE$ of element~$\E$.

Introducing the local discrete bilinear forms
\[
\anE(\sigmaboldn, \tauboldn) := \aE(\Piboldzp \sigmaboldn, \Piboldzp \tauboldn) + \SE( (\Ibold-\Piboldzp) \sigmaboldn, (\Ibold-\Piboldzp) \tauboldn)  \quad\quad \forall \sigmaboldn,\, \tauboldn \in \Sigmaboldn(\E),
\]
we define the global discrete bilinear form
\[
\an(\sigmaboldn, \tauboldn) := \sum_{\E \in \taun} \anE(\sigmaboldn{}_{|\E}, \tauboldn{}_{|\E})	 \quad\quad \forall \sigmaboldn,\, \tauboldn \in \Sigmaboldn.
\]
Following~\cite{Brezzi-Falk-Marini}, the global discrete bilinear form~$\an(\cdot,\cdot)$ is coercive and continuous with constants~$\min(k^*{}^{-1}, \min_{\E \in \taun} \alpha_*)$ and~$\max(k_*{}^{-1}, \max_{\E \in \taun} \alpha^*)$ on the discrete kernel
\begin{equation} \label{discrete:kernel}
\mathcal K_n = \left\{ \tauboldn \in \Sigmaboldn \mid \b(\tauboldn, \vn) = 0 \text{ for all } \vn \in \Vn   \right\}.
\end{equation}
The discrete kernel~$\mathcal K_n$ is contained in the continuous kernel~$\mathcal K$, which is defined as
\[
\mathcal K := \left\{ \taubold \in \Sigmabold \mid \b(\taubold, v) = 0 \text{ for all } v \in V   \right\}.
\]
As for the right-hand side~$(\f, \vn)_{0,\Omega}$, we recall that~$\vn$ is a piecewise polynomial.
Hence, the right-hand side can be approximated at any precision with sufficiently accurate quadrature formulas.

\paragraph*{The virtual element method for the mixed formulation.}  The virtual element method tailored for the approximation of the problem in mixed form~\eqref{mixed:formulation} reads
\begin{equation} \label{VEM:mixed}
\begin{cases}
\text{find } (\sigmaboldn, \un) \in \SigmaboldngN \times \Vn \text{ such that}\\ 
\an(\sigmaboldn, \tauboldn)  +\b(\tauboldn, \un) = -(\gD, \n \cdot \tauboldn)_{0,\GammaD} \quad \forall \tauboldn\in \Sigmaboldnz \\
\b(\sigmaboldn, \vn) = (-\f,\vn)_{0,\Omega} \quad \forall \vn\in \Vn.
\end{cases}
\end{equation}
The definition of the trial and test spaces, and the second equation in~\eqref{VEM:mixed} entail that
\begin{equation} \label{identity:div}
\div \sigmaboldn = \Pizpmo \f \quad \quad \text{in } \Omega.
\end{equation}

\begin{remark}  \label{remark:hp:mixed}
So far, we have discussed the construction of virtual elements with a uniform degree of accuracy over all the elements.
As already highlighted in Remark~\ref{remark:hp:primal}, the flexibility of the virtual element framework allows for the construction of global spaces with variable degrees of accuracy.
Since the construction of spaces with variable degree of accuracy follows along the same lines as of the primal formulation case, we omit the details of the construction.
\eremk
\end{remark}

\subsection{Providing explicit stabilizations} \label{subsection:stabilizations}
Here, we address the issue of providing computable stabilizations with bounds on the constants~$\alphatilde_*$, $\alphatilde^*$, $\alpha_*$, and~$\alpha^*$ that are explicit in terms of degree of accuracy~$\p$;
see~\eqref{local_stab:primal} and~\eqref{local_stab:mixed}, respectively.
We pose ourselves in the situation of a uniform degree of accuracy over all the elements.
The variable case can be tackled as in~\cite{hpVEMcorner}; see also Remarks~\ref{remark:hp:primal} and~\ref{remark:hp:mixed} for further comments on this aspect.

\paragraph*{Theoretical stabilizations.}
A stabilization $\SE(\cdot,\cdot)$ for the primal formulation with explicit bounds on the stabilization constants~$\alphatilde_*$ and~$\alphatilde^*$ can be found in~\cite[Section 4]{hpVEMcorner}.
Assuming that~$\k=1$, such stabilization reads
\begin{equation} \label{stab:primal}
\StildeE(\utilden,\vtilden) = \frac{\p^2}{\hE^2} (\Pitildez \utilden, \Pitildez \vtilden)_{0,\E} + \frac{\p}{\hE} (\utilden, \vtilden)_{0,\partial \E}.
\end{equation}
The following bounds were proven in~\cite[Theorem~2]{hpVEMcorner}:
\begin{equation} \label{stability:bounds:primal}
\alphatilde_*(\p) \gtrsim p^{-5}, \quad \quad \alphatilde^*(\p) \lesssim \p^2.
\end{equation}
Such bounds are extremely crude. As numerically investigated in~\cite[Section~4.1]{hpVEMcorner}, the dependence on~$\p$ is much milder in practice.

It is possible to modify the stabilization in~\eqref{stab:primal} to the instance of variable~$\k$:
\[
\StildeE(\utilden,\vtilden) = \frac{\hE^2}{\p^2} (\k \Pitildez \utilden, \Pitildez \vtilden)_{0,\E} + \frac{\hE}{\p} (\k \utilden, \vtilden)_{0,\partial \E}.
\]
Exploiting the assumptions on~$\k$ in~\eqref{ass:k}, the bounds on the stability constants become
\begin{equation} \label{bounds:stab_primal}
\alphatilde_*(\p) \gtrsim k_* p^{-5}, \quad \quad \alphatilde^*(\p) \lesssim k^*\p^2.
\end{equation}

Next, we focus on an explicit choice for the stabilization~$\SE(\cdot,\cdot)$ in~\eqref{local_stab:mixed} in the mixed VEM~\eqref{VEM:mixed}.
For all~$\E \in \taun$ and~$\sigmaboldn$ and~$\tauboldn$ in the kernel of the projector~$\Piboldzp$, we define the stabilization
\begin{equation} \label{stab:mixed}
\begin{split}
\SE (\sigmaboldn, \tauboldn) 	& = \hE (\k^{-1} \n \cdot \sigmaboldn, \n \cdot \tauboldn  )_{0,\partial \E} \\
& \quad + \hE^2 (\k^{-1} \div\, \sigmaboldn, \div\, \tauboldn)_{0,\E} + \hE^2 (\k^{-1}  \rot\, \sigmaboldn, \rot\, \tauboldn)_{0,\E}.
\end{split}
\end{equation}
Thanks to the choice of the degrees of freedom~\eqref{edgeDOF:flux}--\eqref{orthogonalDOF:flux}, the stabilization~$\SE$ is explicitly computable.
Recall the following four lemmata from~\cite{kvrivzek1984validity, monk2003finite, ncHVEM, hpVEMcorner}, which will be instrumental in the analysis of the stabilization in~\eqref{stab:mixed}.
\begin{lem} \label{lemma:technical1}
For all convex Lipschitz domains~$\E \subset\mathbb R^2$ with diameter~$1$ and~$\taubold \in \HdivE \cap \HcurlE$ with~$\nE \cdot \taubold =0$ on~$\partial \E$, the following bound is valid:
\[
\Vert \taubold \Vert_{0,\E} \lesssim \left( \Vert \div \taubold \Vert_{0,\E} +  \Vert \rot \taubold \Vert_{0,\E}     \right).
\]
\end{lem}
\begin{proof}
See~\cite[Theorem 4.4]{kvrivzek1984validity}.
\end{proof}

\begin{lem} \label{lemma:technical2}
For all simply connected and bounded Lipschitz domains~$\E \subset\mathbb R^2$ with diameter~$1$ and~$\taubold \in \HdivE \cap \HcurlE$ with~$\div \taubold=0$ and~$ \n \cdot \taubold \in L^2(\Omega)$,
the following bound is valid:
\[
\Vert \taubold \Vert_{0,\E} \lesssim \left(  \Vert \rot (\taubold) \Vert_{0,\E} + \Vert \n \cdot \taubold  \Vert_{0,\partial \E}     \right).
\]
\end{lem}
\begin{proof}
This is the two dimensional counterpart of~\cite[Corollary 3.51]{monk2003finite}.
\end{proof}

\begin{lem} \label{lemma:technical3}
Let~$\E$ be a polygon with diameter~$1$. Assume that its edges have a length~$\approx 1$.
Then, for all piecewise polynomials~$\qp$ over~$\partial \E$, the following polynomial inverse inequality is valid:
\begin{equation} \label{Bernardi:estimate}
\Vert \qp \Vert_{0,\partial \E} \lesssim \p^{\frac32} \Vert \qp \Vert_{-\frac12,\partial \E}.
\end{equation}
\end{lem}
\begin{proof}
See~\cite[proof of Theorem~$3.2$]{ncHVEM}.
\end{proof}

\begin{lem} \label{lemma:technical4}
Let~$\E$ be a polygon with diameter~$1$. Assume that its edges have a length~$\approx 1$.
Then, for all~$\qp \in \mathbb P_\p(\E)$, the following polynomial inverse inequality is valid:
\begin{equation} \label{inverse:negative}
\Vert \qp \Vert_{0,\E} \lesssim \p^2 \Vert \qp \Vert_{-1,\E} := \p^2 \sup_{\Phi \in H^1_0(\E), \Phi \ne 0} \frac{(\qp, \Phi)_{0,\E}}{\vert \Phi \vert_{1,\E}}.
\end{equation}
\end{lem}
\begin{proof}
See~\cite[Theorem~5]{hpVEMcorner}.
\end{proof}

We prove the following result.
\begin{prop} \label{proposition:stab:mixed}
For all convex~$\E\in \taun$, the bilinear form~$\SE(\cdot,\cdot)$ in~\eqref{stab:mixed} is such that the following bounds on the constants~$\alpha_*$ and~$\alpha^*$ in~\eqref{local_stab:mixed} are valid:
\begin{equation} \label{bounds:stab_mixed}
\alpha_* \gtrsim (k^*){}^{-1} ,\quad\quad \alpha^*\lesssim (k_*)^{-1} \p^7.
\end{equation}
\end{prop}
\begin{proof}
Throughout the proof, we assume that~$\hE=1$. The general case follows from a scaling argument.
Besides, it suffices to prove the statement for~$\k =1$. The general case follows from~\eqref{ass:k}.
\medskip

First, we show the bound on~$\alpha_*$.
Given~$\tauboldn \in \Sigmaboldn(\E)$, we define the functions~$\tauboldtilden$ and~$\tauboldboundaryn$ as
\begin{equation} \label{splitting:1}
\nE \cdot \tauboldtilden  = 0, \quad \div \tauboldtilden = \div \tauboldn,\quad \rot \tauboldtilden = \rot \tauboldn
\end{equation}
and
\begin{equation} \label{splitting:2}
\nE \cdot \tauboldboundaryn  =  \nE \cdot \tauboldn , \quad \div \tauboldboundaryn = 0,\quad \rot \tauboldboundaryn = 0.
\end{equation}
Clearly, $\tauboldn = \tauboldtilden + \tauboldboundaryn$.
Using Lemmata~\ref{lemma:technical1} and~\ref{lemma:technical2}, we deduce
\[
\begin{split}
\Vert \tauboldn \Vert_{0,\E}
& \le \Vert \tauboldtilden \Vert_{0,\E} + \Vert \tauboldboundaryn \Vert_{0,\E} \\
& \lesssim \Vert \rot \tauboldboundaryn \Vert_{0,\E} + \Vert \nE \cdot \tauboldboundaryn \Vert_{0,\partial \E} + \Vert \rot \tauboldtilden \Vert_{0, \E} + \Vert \div \tauboldtilden \Vert_{0, \E}. \\
\end{split}
\]
Identities~\eqref{splitting:1} and~\eqref{splitting:2} imply
\[
\Vert \tauboldn \Vert_{0,\E} \lesssim \Vert \nE \cdot \tauboldn \Vert_{0,\partial \E} + \Vert \div \tauboldn \Vert_{0,\E} + \Vert \rot \tauboldn \Vert_{0,\E},
\]
which is the desired bound.
\medskip

Next, we deal with the bound on~$\alpha^*$. It suffices to show a bound for the terms on the right-hand side of~\eqref{stab:mixed} by some constant depending on~$\p$ times~$\Vert \tauboldn \Vert_{0,\E}$.
Since~$\nE \cdot \taubold$ is a piecewise polynomial on~$\partial \E$, 
we use Lemma~\ref{lemma:technical3} and get
\begin{equation} \label{inverse:inequality:1/2}
\Vert \nE \cdot  \tauboldn \Vert_{0,\partial \E} \lesssim \p^{\frac32} \, \Vert \nE \cdot  \tauboldn \Vert_{-\frac{1}{2}, \partial \E}.
\end{equation}
Using~\eqref{inverse:inequality:1/2} and the inequality~\cite[Theorem 3.24]{monk2003finite}, we obtain
\[
\Vert \nE \cdot  \tauboldn \Vert_{0,\partial \E} \lesssim \p^{\frac32} \left( \Vert \tauboldn \Vert_{0,\E} + \Vert \div \tauboldn \Vert_{0,\E}   \right).
\]
We show an upper bound on the divergence term, i.e., the second term appearing on the right-hand side of~\eqref{stab:mixed}. 
Use Lemma~\ref{lemma:technical4} substituting~$\qp$ with~$\div \tauboldn$ in~\eqref{inverse:negative}, and an integration by parts, to get
\[
\begin{split}
\Vert \div \tauboldn \Vert_{0,\E} 	& \lesssim \p^2 \Vert \div (\tauboldn) \Vert_{-1,\E} := \p^2 \sup_{\Phi \in H^1_0(\E), \Phi \ne 0} \frac{(\div \tauboldn, \Phi)_{0,\E}}{\vert \Phi \vert_{1,\E}}\\
						& = \p^2 \sup_{\Phi \in H^1_0(\E), \Phi \ne 0} \frac{(\tauboldn, \nabla \Phi)_{0,\E}}{\vert \Phi \vert_{1,\E}} \le \p^2 \Vert \tauboldn \Vert_{0,\E}.
\end{split}
\]
This concludes the proof of the upper bound on the first two terms on the right-hand side of~\eqref{stab:mixed}.

Finally, we use Lemma~\ref{lemma:technical4} substituting~$\qp$ with~$\rot \tauboldn$ to show the upper bound on the third term on the right-hand side of~\eqref{stab:mixed}.
Denoting the vectorial rotor by \textbf{curl}, we can write
\[
\begin{split}
\Vert \rot \tauboldn \Vert_{0,\E} 	&\lesssim \p^2 \Vert \rot \tauboldn \Vert_{-1,\E} := \p^2 \sup_{\Phi \in H^1_0(\E), \Phi \ne 0} \frac{(\rot \tauboldn, \Phi)_{0,\E}}{\vert \Phi \vert_{1,\E}}\\
							& = \p^2 \sup_{\Phi \in H^1_0(\E), \Phi \ne 0} \frac{(\tauboldn, \curlbold\, \Phi)_{0,\E}}{\vert \Phi \vert_{1,\E}}  \lesssim \p^2 \Vert \tauboldn \Vert_{0,\E}, \\ 
\end{split}
\]
whence the assertion follows.
\end{proof}
Note that the assumption on the convexity of the elements  in Proposition~\ref{proposition:stab:mixed} is needed only to apply Lemma~\ref{lemma:technical1}.
As for the $\h$-version of the method, the stabilization~$\SE$ in~\eqref{stab:mixed} can be employed as well.

\paragraph*{Practical stabilizations.}
In the numerical experiments of Sections~\ref{section:nr} and~\ref{section:nr:local}, we shall not employ only stabilizations~\eqref{stab:primal} and~\eqref{stab:mixed}.
Rather, we suggest to use variants of the so-called D-recipe stabilization; see~\cite{VEM3Dbasic}.
In fact, as analyzed in~\cite{fetishVEM, fetishVEM3D}, the D-recipe leads to an extremely robust performance of method~\eqref{VEM:primal}, and is straightforward to implement.

We employ the following stabilization for the primal formulation:  given~$\NVtildeE := \dim(\VtildenE)$,
for all~$\E \in \taun$, given~$\{ \varphitilde_j \}_{j=1} ^{\NVtildeE}$ the canonical basis of the local space~$\VtildenE$,
\begin{equation} \label{practical:stab:primal}
\StildeE(\varphitilde_j, \varphitilde_\ell) =  \max \left( \vert \k \vert , (\k  \nabla \Pinablap \varphitilde_j, \nabla \Pinablap \varphitilde_\ell  )_{0,\E}    \right)    \delta_{j,\ell}    \quad \quad \forall j,\ell=1,\dots, \NVtildeE,
\end{equation}
where $\vert \cdot \vert$ denotes the evaluation of a given scalar function at the barycenter of~$\E$.

Here, $\delta_{j, \ell}$ denotes the Kronecker delta, whereas the projector~$\Pinablap$ is defined in~\eqref{projection:nabla}.
\medskip

As for the D-recipe stabilization for the mixed formulation, we employ the following: given $\NsigmaboldnE := \dim(\SigmaboldnE)$,
for all~$\E \in \taun$, given~$\{ \varphibold_j \}_{j=1} ^{\NsigmaboldnE}$ the canonical basis of the local space~$\SigmaboldnE$,
\begin{equation} \label{practical:stab:mixed}
\SE(\varphibold_j, \varphibold_\ell) =  \max \left( \vert \k^{-1} \vert \; \hE^2, (\k^{-1} \Piboldzp \varphibold_j, \Piboldzp \varphibold_\ell  )_{0,\E}   \right)    \delta_{j,\ell}    \quad \quad \forall j,\ell=1,\dots,\NsigmaboldnE,
\end{equation}
where the projector~$\Piboldzp$ is defined in~\eqref{projector:mixed}.

\begin{remark}
The difficulty in providing $\p$-explicit bounds for the two practical stabilizations is related to the fact that we need to keep track of the dependence in terms of~$\p$, which is possible to do when recovering via integration by parts some polynomial terms.
For the practical stabilization for the primal formulation, this could be done under more assumptions on the polynomial basis used in the definition of the bulk degrees of freedom~\eqref{internal:moments:primal}, as for the primal formulation; see~\cite[Theorem~$2.3$]{pVEMmultigrid} for more details.
Instead, for the practical stabilization for the mixed formulation, we do not know how to prove $\p$-explicit bounds.
\end{remark}

\medskip
In Section~\ref{section:nr} below, we show the numerical experiments for the adaptive method employing the practical stabilizations.
However, we present some numerics comparing the practical and theoretical stabilizations for the $\p$-version a priori mixed VEM.
Moreover, we numerically demonstrate that the convexity assumption in Proposition~\ref{proposition:stab:mixed} is only a theoretical artefact with no effect whatsoever on the practical behaviour of the method.

\subsection{Well-posedness of the two virtual element methods} \label{subsection:well}
In this section, we prove the well-posedness of the primal and mixed methods~\eqref{VEM:primal} and~\eqref{VEM:mixed}.

The well-posedness of the primal method~\eqref{VEM:primal} follows from the continuity and coercivity of the discrete bilinear form~$\atilden(\cdot, \cdot)$,
a lifting argument,
the continuity of the discrete right-hand side, and the Lax-Milgram lemma.

As for the mixed VEM formulation~\eqref{VEM:mixed},
in addition to the usual lifting argument,
we need two ingredients in order to prove the well-posedness of the method.
The first one is the continuity and the coercivity of the bilinear form~$\an(\cdot, \cdot)$ on the discrete kernel~\eqref{discrete:kernel}.
The second one is the validity of the inf-sup condition for the bilinear form~$\b(\cdot, \cdot)$ with explicit bounds on the inf-sup constant in terms of~$\h$ and~$\p$.
The remainder of the section is devoted to prove such an inf-sup condition.
\begin{thm} \label{theorem:inf-sup}
There exists a constant~$\betan > 0$ independent of the discretization parameters, such that for all~$\vn \in \Vn$ there exists $\tauboldn \in \Sigmaboldnz$ satisfying
\begin{equation} \label{discrete:inf-sup}
\b(\vn,\tauboldn) \ge \betan \Vert \tauboldn \Vert_{\div,\Omega} \Vert \vn \Vert _{0,\Omega}.
\end{equation}
The constant~$\betan$ is known in closed form:
\[
\betan = \left( \cOmega \frac{\vert \Omega \vert^{\frac{1}{2}}}{\vert \GammaD \vert^{\frac{1}{2}}} +\cOmega + 1  \right)^{-1},
\]
where~$\cOmega$ is a constant depending only on the shape of~$\Omega$, which will be detailed in the proof.
\end{thm}
\begin{proof}
With each~$\vn \in \Vn$, we associate a function~$\tauboldn \in \Sigmaboldnz$ as follows:
for each~$\E \in \taun$, introduce a local~$\tauboldn$ satisfying
\[
\begin{cases}
\div \tauboldn = \vn 				    & \text{in } \E\\
\rot \tauboldn = 0 					    & \text{in } \E\\
\nE \cdot  \tauboldn  = \ctaubolde       & \text{on } \e \quad \forall \e \in \EE,\; \e \not \subset \GammaD \cup \GammaN \\
\nE \cdot  \tauboldn  = \ctauboldGammaD  & \text{on } \e \quad \forall \e \subset\GammaD\\
\nE \cdot  \tauboldn  = 0                & \text{on } \e \quad \forall \e \subset\GammaN.
\end{cases}
\]
Above, $\ctauboldGammaD$ is defined as the following global constant over~$\GammaD$:
\begin{equation} \label{ctau}
\ctaubold = \frac{1}{\vert \GammaD \vert} \int_{\Omega} \vn,
\end{equation}
whereas~$\ctaubolde$ are piecewise constant functions over the interior skeleton such that two properties are satisfied:
(\emph{i}) the compatability conditions of the above problems are satisfied;
(\emph{ii}) $\ctaubolde$ is single valued on each interior edge~$\e$.
It is possible to fix such constant values as an easy consequence of Gauss' formula for graphs.

This second property entails that we can define a global~$\tauboldn$ in~$\Sigmaboldnz$ as the solution to the global div-rot problem
\begin{equation} \label{div-rot:problem}
\begin{cases}
\div \tauboldn = \vn 				& \text{in } 	\Omega\\
\rot \tauboldn = 0 					& \text{in } 	\Omega\\
\nE \cdot  \tauboldn  = \ctaubold	& \text{on } 	\GammaD\\
\nE \cdot  \tauboldn  = 0	        & \text{on } 	\GammaN.\\
\end{cases}
\end{equation}
Applying~\cite[Remark p.~367]{costabel1990remark} and the Friedrichs' inequality~\cite[Corollary~3.51]{monk2003finite} yields
\[
\Vert \tauboldn \Vert_{0,\Omega} \le \cOmega \left(   \Vert \n \cdot  \tauboldn  \Vert_{0,\partial \Omega} + \Vert \div \tauboldn \Vert_{0,\Omega}   \right),
\]
where~$\cOmega$ is a positive constant depending only on~$\Omega$.

We deduce
\[
\begin{split}
\Vert \tauboldn \Vert_{0,\Omega} 	& \le \cOmega \left( \Vert \n \cdot  \tauboldn  \Vert_{0,\GammaD} + \Vert \div \tauboldn  \Vert_{0,\Omega}    \right)  \overset{\eqref{div-rot:problem}}{=} \cOmega \left( \vert \GammaD \vert^{\frac{1}{2}} \vert \ctaubold \vert + \Vert \div \tauboldn \Vert_{0,\Omega}    \right)\\
& \overset{\eqref{ctau}}{=} \cOmega \left(  \frac{1}{\vert \GammaD \vert^{\frac{1}{2}}} \left\vert  \int_\Omega \vn  \right\vert  + \Vert \div \tauboldn \Vert_{0,\Omega}   \right)\\
& \overset{\eqref{div-rot:problem}}{=} \cOmega \left(  \frac{1}{\vert \GammaD \vert^{\frac{1}{2}}} \left\vert  \int_\Omega \div \tauboldn \right\vert  + \Vert \div \tauboldn \Vert_{0,\Omega}   \right)
\le \cOmega \left( \frac{\vert \Omega \vert^{\frac{1}{2}}}{\vert \GammaD \vert^{\frac{1}{2}}} +1  \right)  \Vert \div \tauboldn  \Vert_{0,\Omega}.\\
\end{split}
\]
This implies
\begin{equation} \label{bound:denominator}
\Vert \tauboldn \Vert_{\div, \Omega} \le \left( \cOmega \frac{\vert \Omega \vert^{\frac{1}{2}}}{\vert \GammaD \vert^{\frac{1}{2}}} + \cOmega +1     \right) \Vert \div \tauboldn \Vert_{0,\Omega} =: \betan^{-1} \Vert \div \tauboldn \Vert_{0,\Omega}.
\end{equation}
Note that
\[
\Vert \vn \Vert_{0,\Omega} = \frac{(\vn,\vn)_{0,\Omega}}{\Vert \vn \Vert_{0,\Omega}} \overset{\eqref{div-rot:problem}}{=} \frac{(\vn, \div \tauboldn)_{0,\Omega}}{\Vert \div \tauboldn \Vert_{0,\Omega}}.
\]
Applying~\eqref{bound:denominator} to this identity, we deduce the discrete inf-sup condition
\[
\Vert \vn \Vert_{0,\Omega} \le \frac{1}{\betan} \frac{(\vn, \div \tauboldn)_{0,\Omega}}{\Vert \tauboldn \Vert_{\div, \Omega}}.
\]
\end{proof}
The discrete inf-sup condition~\eqref{discrete:inf-sup}, together with the definition of the continuity of the discrete bilinear forms~$\an(\cdot, \cdot)$ and~$\b(\cdot, \cdot)$,
and the coercivity of~$\an(\cdot, \cdot)$ on the discrete kernel~$\mathcal K_n$~\eqref{discrete:kernel}, is sufficient to prove the well-posedness of method~\eqref{VEM:mixed}; see~\cite{BrezziFortin}.
In fact, a lifting argument allows us to seek solutions in the discrete space~$\Sigmaboldnz$, i.e., solutions with zero normal trace on~$\GammaN$.

\begin{remark} \label{remark:inf_sup:BFM}
We have proved Theorem~\ref{theorem:inf-sup} assuming that~$\GammaD \ne \emptyset$, which is an assumption stipulated in Section~\ref{section:introduction}.
The case of pure Neumann boundary conditions has to be dealt with slightly differently. The primal virtual element space has to be endowed with a zero average constraint.
In order to prove the inf-sup condition one should proceed as in \cite[Theorem~4.2 and Corollary~4.3]{Brezzi-Falk-Marini}.
Besides, one ought to prove that the best approximant in mixed virtual element spaces converges optimally in terms of~$\h$ and~$\p$ to a target function, so that the inf-sup constant is $\p$-robust.
This can be indeed proven by defining the best approximation in virtual element spaces as the degrees of freedom interpolant of the continuous function to approximate.
However, whatever boundary conditions we pick, the discrete inf-sup constant is $\p$-robust.
To the best of our knowledge, this is not the case in the discontinuous Galerkin setting for mixed problems; see, e.g., \cite[Section~4.2]{SchoetzauWihler_hpMixed}.
Of course, the price to pay is the $\p$ dependence in the stability estimates; see~\eqref{stability:bounds:primal} and Proposition~\ref{proposition:stab:mixed}.
\eremk
\end{remark}

\section{The hypercircle method for the VEM} \label{section:apos}
The aim of the present section is to construct an equilibrated a posteriori error estimator and to prove lower and upper bounds of such error estimator in terms of the exact error. 
In Section~\ref{subsection:identity}, we show an identity, which is the basic tile of the a posteriori error analysis and exhibit the equilibrated error estimator.
Next, in Sections~\ref{subsection:upper_bound} and~\ref{subsection:lower_bound}, we show its reliability and efficiency.

Throughout, we assume that each local space has a fixed degree of accuracy~$\p$.
The case of variable degree of accuracy is dealt with by substituting~$\p$ with a local~$\pE$ on each element~$\E$ of~$\taun$.
This is reflected in estimates~\eqref{global:A} and~\eqref{global:B}, as well as in the definition of the two stabilizations in~\eqref{practical:stab:primal} and~\eqref{practical:stab:mixed}.
In the residual estimator setting it is mandatory to demand that neighbouring elements have comparable degree of accuracy, see~\cite[equation (4)]{hpVEMapos}.
This condition is not needed in the analysis contained in the present paper.
For more details on the variable degree case, see also Remarks~\ref{remark:hp:primal} and~\ref{remark:hp:mixed}.

\subsection{The equilibrated a posteriori error estimator} \label{subsection:identity}
Given~$\utilde$, $\sigmabold$, $\utilden$, and $\sigmaboldn$ the solutions to~\eqref{primal:formulation}, \eqref{mixed:formulation}, \eqref{VEM:primal}, and~\eqref{VEM:mixed}, respectively,
we observe that
\begin{equation} \label{identity:contazzi}
\Vert \k^{\frac{1}{2}} (\nabla \utilde - \nabla \utilden) \Vert^2_{0,\Omega} + \Vert \k^{-\frac{1}{2}}(\sigmabold - \sigmaboldn) \Vert^2_{0,\Omega} = \Vert \k^{\frac{1}{2}}\nabla \utilden + \k^{-\frac{1}{2}}\sigmaboldn \Vert^2_{0,\Omega} -2 \int_\Omega \nabla (\utilde - \utilden) \cdot (\sigmabold - \sigmaboldn).
\end{equation}
In order to show this identity, we observe that~\eqref{strong:identity} entails
\[
\begin{split}
& \Vert \k^{\frac{1}{2}}\nabla \utilden + \k^{-\frac{1}{2}}\sigmaboldn \Vert^2_{0,\Omega} -2 \int_\Omega \nabla (\utilde - \utilden) \cdot (\sigmabold - \sigmaboldn)\\
&= \int_\Omega \k^{\frac{1}{2}} (\nabla \utilden - \nabla \utilde) \cdot \left( \k^{\frac{1}{2}} \nabla \utilden + \k^{-\frac{1}{2}} \sigmaboldn  \right) + \int_\Omega (  \k^{-\frac{1}{2}} \sigmaboldn + \k^{\frac{1}{2}}\nabla \utilde) \cdot \left( \k^{\frac{1}{2}} \nabla \utilden + \k^{-\frac{1}{2}} \sigmaboldn  \right)\\
& \quad - 2 \int_\Omega \k^{\frac{1}{2}} \nabla (\utilde - \utilden) \cdot  \k^{-\frac{1}{2}} (\sigmabold - \sigmaboldn)\\
& = \int_\Omega \k^{\frac{1}{2}} (\nabla \utilde - \nabla \utilden) \cdot \left( -\k^{\frac{1}{2}} \nabla \utilden - \k^{-\frac{1}{2}} \sigmaboldn + \k^{\frac{1}{2}} \nabla \utilde + \k^{-\frac{1}{2}} \sigmaboldn    \right)\\
& \quad + \int_\Omega \k^{-\frac{1}{2}} (\sigmabold - \sigmaboldn) \cdot \left( -\k^{\frac{1}{2}} \nabla \utilden - \k^{-\frac{1}{2}} \sigmaboldn + \k^{-\frac{1}{2}} \sigmabold + \k^{\frac{1}{2}} \nabla \utilden    \right)\\
& = \int_\Omega \k (\nabla \utilde - \nabla \utilden) \cdot  (\nabla \utilde - \nabla\un) + \int_\Omega \k^{-1} (\sigmabold - \sigmaboldn) \cdot (\sigmabold - \sigmaboldn),
\end{split}
\]
which is~\eqref{identity:contazzi}.

We rewrite the last term on the right-hand side of~\eqref{identity:contazzi} using an elementwise integration by parts: for all~$\E \in \taun$,
\[
- \int_\E \nabla (\utilde - \utilden) \cdot (\sigmabold - \sigmaboldn) = \int_\E (\utilde - \utilden) \,\div(\sigmabold - \sigmaboldn) - \int_{\partial \E} (\utilde - \utilden) \, \n \cdot(\sigmabold - \sigmaboldn) .
\]
We reshape the first term on the right-hand side.
Recalling that~$\div \sigmabold = \f$ and~$\div \sigmaboldn = \Pizpmo \f$, see~\eqref{identity:div},
we write
\[
\int_\E (\utilde - \utilden) \,(\div(\sigmabold - \sigmaboldn)) = \int_\E (u-\utilden) (\f-\Pizpmo \f) .
\]
We collect all the contributions and deduce that, for every piecewise discontinuous polynomial~$\qpmo$ of degree~$\p-1$ over~$\taun$,
\begin{equation} \label{equation:star}
\begin{split}
&\Vert \k^{\frac{1}{2}}\nabla (\utilde-\utilden) \Vert_{0,\Omega}^2 + \Vert \k^{-\frac{1}{2}} (\sigmabold - \sigmaboldn) \Vert^2_{0,\Omega} \\
& = \Vert \k^{\frac{1}{2}}\nabla \utilden + \k^{-\frac{1}{2}}\sigmaboldn \Vert_{0,\Omega}^2 + 2\int_\Omega (\utilde - \utilden - \qpmo) (\f - \Pizpmo \f)  \\
& \quad - 2 \sum_{\E \in \taun} \int_{\partial \E} (\utilde - \utilden) \n \cdot(\sigmabold- \sigmaboldn).
\end{split}
\end{equation}
The internal interface contributions appearing in the last term on the right-hand side of~\eqref{equation:star} are zero. The virtual element spaces have been tailored so that this property is fulfilled.
Furthermore, the boundary contributions disappear thanks to the assumption (\textbf{D}).

Thus, we write
\begin{equation} \label{important_equation:apos}
\begin{split}
&\Vert \k^{\frac{1}{2}}\nabla (\utilde-\utilden) \Vert_{0,\Omega}^2 + \Vert \k^{-\frac{1}{2}} (\sigmabold - \sigmaboldn) \Vert^2_{0,\Omega} \\
& = \Vert \k^{\frac{1}{2}}\nabla \utilden + \k^{-\frac{1}{2}}\sigmaboldn \Vert_{0,\Omega}^2 + 2\int_\Omega (\utilde - \utilden - \qpmo) (\f - \Pizpmo \f) . \\
\end{split}
\end{equation}
In the two forthcoming Sections~\ref{subsection:upper_bound} and~\ref{subsection:lower_bound},
we show upper and lower bounds on the right-hand side of~\eqref{important_equation:apos}. This will give rise to a natural choice for the equilibrated error estimator.
Henceforth, we refer to the square root of the left-hand side of~\eqref{important_equation:apos} as to the exact error of the method.
In particular, the exact error is the square root of the sum of the square of the error of the primal~\eqref{VEM:primal} and mixed~\eqref{VEM:mixed} methods.

\paragraph*{The error estimator.}
Since~$\Vert \k^{\frac{1}{2}}\nabla \utilden + \k^{-\frac{1}{2}}\sigmaboldn \Vert_{0,\Omega}$ is not computable, we propose the local
\begin{equation} \label{local_EE}
\begin{split}
\etahypE ^2	& := \Vert \k^{\frac{1}{2}} \nabla \Pinablap \utilden + \k^{-\frac{1}{2}} \Piboldzp \sigmaboldn \Vert^2_{0,\E} \\
				& \quad +  \left[ \StildeE((I-\Pinablap)\utilden, (I-\Pinablap)\utilden)  + \SE( (\Ibold-\Piboldzp)\sigmaboldn, (\Ibold-\Piboldzp)\sigmaboldn )  \right],\\
\end{split}
\end{equation}
and global error estimators
\begin{equation} \label{global_EE}
\etahyp^2 : = \sum_{\E \in \taun} \etahypE^2.
\end{equation}

\subsection{Reliability} \label{subsection:upper_bound}
From~\eqref{important_equation:apos}, we deduce the following upper bound:
\[
\begin{split}
& \Vert \k^{\frac{1}{2}}\nabla (\utilde - \utilden) \Vert^2_{0,\Omega} 	+ \Vert \k^{-\frac{1}{2}} (\sigmabold - \sigmaboldn) \Vert^2_{0,\Omega} \\
& \le 2 \big( \Vert \k^{\frac{1}{2}} (\nabla \utilden - \nabla \Pinablap \utilden) \Vert^2_{0,\Omega} + \Vert \k^{\frac{1}{2}}  \nabla \Pinablap \utilden + \k^{-\frac{1}{2}}\Piboldzp \sigmaboldn \Vert^2_{0,\Omega} + \Vert \k^{-\frac{1}{2}} (\Piboldzp \sigmaboldn - \sigmaboldn) \Vert^2_{0,\Omega} \big) \\
& \quad  +2 \int_\Omega (\utilde - \utilden - \qpmo) (\f - \Pizpmo \f) \quad \quad \forall \qpmo \in \mathcal S^{\p-1,-1} (\Omega, \taun).
\end{split}
\]
Set~$\qpmo$ as the piecewise~$L^2$ projection of~$u- \utilden$ over~$\taun$. Using standard~$\h$- and~$\p$- approximation estimates, we deduce
\[
\begin{split}
\int_\Omega 	& (\utilde - \utilden - \qpmo) (\f - \Pizpmo \f)  \le \sum_{\E \in \taun} \Vert \utilde - \utilden - \qpmo \Vert_{0,\E} \Vert \f - \Pizpmo \f \Vert_{0,\E} \\
			& \le \max_{\E \in \taun} \cB(\E)   \sum_{\E \in \taun} \frac{\hE}{\p} \vert \utilde - \utilden \vert_{1,\E} \Vert \f - \Pizpmo \f \Vert_{0,\E}\\
			& \le \max_{\E \in \taun} \cB(\E)  \vert \utilde - \utilden \vert_{1,\Omega} \left( \sum_{\E \in \taun} \frac{\hE^2}{\p^2} \Vert \f - \Pizpmo \f \Vert_{0,\E} ^2  \right)^{\frac{1}{2}}, \\
\end{split}
\]
where~$\cB(\E)$ is the positive constant appearing in the~$\h$- and~$\p$-approximation estimates of, e.g., \cite[Lemma~4.5]{babuskasurihpversionFEMwithquasiuniformmesh}.

Young's inequality entails
\[
\begin{split}
& \int_\Omega	(\utilde - \utilden - \qpmo) (\f - \Pizpmo \f)  \\
& \le 4 \varepsilon  \left( \max_{\E \in \taun} \cB(\E)  \right)^2 \vert \utilde - \utilden \vert_{1,\Omega}^2 + \frac{1}{4 \varepsilon}  \sum_{\E \in \taun} \frac{\hE^2}{\p^2} \Vert \f - \Pizpmo \f \Vert_{0,\E} ^2   \\
& \le 4 k_*^{-\frac{1}{2}} \varepsilon  \left( \max_{\E \in \taun} \cB(\E)  \right)^2 \Vert \k^{\frac{1}{2}} \nabla (\utilde - \utilden) \Vert_{0,\Omega}^2 + \frac{1}{4 \varepsilon}  \sum_{\E \in \taun} \frac{\hE^2}{\p^2} \Vert \f - \Pizpmo \f \Vert_{0,\E} ^2  \\
\end{split}
\]
for all~$\varepsilon >0$, where~$c_B(\E)$ denotes the best $\h\p$-approximation constant on element~$\E$.

Recalling~\eqref{strong:identity} and setting
\[ \varepsilon = \frac{k_*^{\frac{1}{2}}}{16 (\max_{\E \in \taun} c_B(\E) )^2},\]
we obtain
\[
\begin{split}
& \Vert \k^{\frac{1}{2}}\nabla (\utilde - \utilden) \Vert^2_{0,\Omega} 	+ \Vert \k^{-\frac{1}{2}} (\sigmabold - \sigmaboldn) \Vert^2_{0,\Omega} \\
& \le 4 \left(  \Vert \k^{\frac{1}{2}} (\nabla \utilden -  \nabla \Pinablap\utilden) \Vert^2_{0,\Omega}   + \Vert \k^{\frac{1}{2}}  \nabla \Pinablap\utilden + \k^{-\frac{1}{2}}\Piboldzp \sigmaboldn \Vert_{0,\Omega}^2         \right.\\
& \quad \quad \left. + \Vert \k^{-\frac{1}{2}} (\Piboldzp \sigmaboldn - \sigmaboldn ) \Vert_{0,\Omega}^2  \right) + 16 \frac{(\max_{\E \in \taun} \cB(\E))^2}{k_*^{\frac{1}{2}}} \sum_{\E \in \taun} \frac{\hE^2}{\p^2} \Vert \f - \Pizpmo \f \Vert^2_{0,\E}\\
& \le 4  \Vert \k^{\frac{1}{2}} \nabla \Pinablap \utilden + \k^{-\frac{1}{2}}\Piboldzp \sigmaboldn \Vert_{0,\Omega}^2 \\
& \quad + 4   \sum_{\E \in \taun} \max(\alphatilde_*^{-1}, \alpha_*^{-1}) \left[ \StildeE ( (I-\Pinablap) \utilden, (I-\Pinablap) \utilden)  + \SE( (\Ibold - \Piboldzp) \sigmaboldn , (\Ibold - \Piboldzp) \sigmaboldn   )     \right]\\
& \quad + 16  \frac{(\max_{\E \in \taun} \cB(\E))^2}{k_*^{\frac{1}{2}}} \sum_{\E \in \taun} \frac{\hE^2}{\p^2} \Vert \f - \Pizpmo \f \Vert^2_{0,\E} . \\
\end{split}
\]
Eventually, \eqref{identity:div} entails
\begin{equation} \label{reliability:global}
\begin{split}
&\Vert \k^{\frac{1}{2}}\nabla (\utilde - \utilden) \Vert^2_{0,\Omega} 	+ \Vert \k^{-\frac{1}{2}} (\sigmabold - \sigmaboldn) \Vert^2_{0,\Omega} \le  4 \Vert \k^{\frac{1}{2}} \nabla \Pinablap \utilden + \k^{-\frac{1}{2}}\Piboldzp \sigmaboldn \Vert_{0,\Omega}^2  \\
& \quad +  4 \sum_{\E \in \taun} \max(\alphatilde_*^{-1}, \alpha_*^{-1})  \left( \StildeE ( (I-\Pinablap) \utilden, (I-\Pinablap) \utilden)  + \SE( (\Ibold - \Piboldzp) \sigmaboldn , (\Ibold - \Piboldzp) \sigmaboldn   )     \right) \\
& \quad +  16 \frac{(\max_{\E \in \taun} \cB(\E))^2}{k_*^{\frac{1}{2}}} \sum_{\E \in \taun} \frac{\hE^2}{\p^2} \Vert \f - \div \sigmaboldn \Vert^2_{0,\E} . \\
\end{split}
\end{equation}
All the terms on the right-hand side are computable with the exception of the oscillation of the right-hand side. This term can be approximated at any precision employing a sufficiently accurate quadrature formula.
\medskip 

We have proven the following reliability result.
\begin{thm} \label{theorem:bounds:reliability}
Let the assumptions~(\textbf{G1}), (\textbf{G2}), (\textbf{K}), and (\textbf{D}) be valid.
Let~$\utilde$, and $u$ and~$\sigmabold$ be the solutions to~\eqref{primal:formulation} and~\eqref{mixed:formulation},
and~$\utilden$, and $\un$ and~$\sigmaboldn$ be the solutions to~\eqref{VEM:primal} and~\eqref{VEM:mixed}, respectively.
The following bound on the exact error in terms of the equilibrated error estimator and oscillation terms is valid:
\begin{equation} \label{global:A}
\begin{split}
&\Vert \k^{\frac{1}{2}}\nabla (\utilde - \utilden) \Vert^2_{0,\Omega} 	+ \Vert \k^{-\frac{1}{2}} (\sigmabold - \sigmaboldn) \Vert^2_{0,\Omega} \\
& \le 4 \sum_{\E \in \taun} \max(\alphatilde_*^{-1}, \alpha_*^{-1})  \etahypE^2 + 16 \frac{(\max_{\E \in \taun} \cB(\E))^2}{k_*^{\frac{1}{2}}} \sum_{\E \in \taun} \frac{\hE^2}{\p^2} \Vert \f - \div \sigmaboldn \Vert^2_{0,\E} . \\
\end{split}
\end{equation}
Bound~\eqref{global:A} is fully explicit in terms of~$\h$ and~$\p$ where the $\p$-dependence is also contained in the stabilization constants.
\end{thm}
\begin{remark}
In the standard finite element setting, see, e.g., \cite{braess2009equilibrated}, the first term on the right-hand side of~\eqref{reliability:global} reads
\[
\Vert \k^{\frac{1}{2}}\nabla \utilden + \k^{-\frac{1}{2}}\sigmaboldn \Vert_{0,\Omega},
\]
whereas the second and the third vanish. The fourth term is a higher-order oscillation term.
\end{remark}

\subsection{Efficiency} \label{subsection:lower_bound}
Here, we show an upper bound on the local equilibrated error estimator~$\etahypE$~\eqref{local_EE} in terms of the exact error. In other words, we prove the efficiency of the local equilibrated error estimator.

First, we focus on the norm of the projected discrete solutions, i.e., the first term on the right-hand side of~\eqref{local_EE}.
For all polygons~$\E$, using~\eqref{strong:identity} and the stability of orthogonal projectors, we get
\[
\begin{split}
\Vert \k^{\frac{1}{2}} \nabla \Pinablap \utilden + \k^{-\frac{1}{2}} \Piboldzp \sigmaboldn \Vert_{0,\E}^2	& \le 2 \left( \Vert \k^{\frac{1}{2}} \nabla \utilde -  \nabla \Pinablap \utilden \Vert_{0,\E}^2 + \Vert \k^{-\frac{1}{2}} (\sigmabold - \Piboldzp \sigmaboldn) \Vert^2_{0,\E} \right)\\
												& \le 4 \big( \Vert \k^{\frac{1}{2}} (\nabla \utilde - \nabla \Pinablap \utilde) \Vert^2_{0,\E} + \Vert \k^{\frac{1}{2}}  \nabla \Pinablap (\utilde-\utilden)  \Vert _{0,\E}^2  \\
												& \quad \quad + \Vert \k^{-\frac{1}{2}} (\sigmabold - \Piboldzp \sigmabold) \Vert^2_{0,\E} + \Vert \k^{-\frac{1}{2}} \Piboldzp (\sigmabold - \sigmaboldn) \Vert^2_{0,\E} \big)\\
												& \le 4 \big( \Vert \k^{\frac{1}{2}} (\nabla \utilde - \nabla \Pinablap\utilde) \Vert^2_{0,\E} + \Vert \k^{\frac{1}{2}} \nabla (\utilde-\utilden)  \Vert _{0,\E}^2  \\
												& \quad \quad + \Vert \k^{-\frac{1}{2}} (\sigmabold - \Piboldzp \sigmabold ) \Vert^2_{0,\E} + \Vert \k^{-\frac{1}{2}} (\sigmabold - \sigmaboldn )\Vert^2_{0,\E} \big).\\
\end{split}
\]
Next, we deal with the stabilization terms, i.e., the second and third terms on the right-hand side of~\eqref{local_EE}. We begin with the stabilization term stemming from the discretization of the primal formulation:
\[
\begin{split}
\StildeE ( (I-\Pinablap) \utilden, 	& (I-\Pinablap) \utilden) \le \alphatilde^* \Vert \k^{\frac{1}{2}} \nabla (I-\Pinablap) \utilden \Vert^2_{0,\E}\\
						& \le 2 \alphatilde^* \left[ \Vert \k^{\frac{1}{2}} \nabla (\utilde-\utilden) \Vert^2_{0,\E} + \Vert \k^{\frac{1}{2}}\nabla (\utilde-\Pinablap \utilde) \Vert^2_{0,\E}+ \Vert \k^{\frac{1}{2}}\nabla \Pinablap (\utilde-\utilden) \Vert^2_{0,\E} \right]\\
						& \le 4 \alphatilde^* \Vert \k^{\frac{1}{2}}\nabla (\utilde - \utilden) \Vert^2_{0,\E} + 2 \alphatilde^* \Vert \k^{\frac{1}{2}}\nabla (\utilde - \Pinablap u) \Vert^2_{0,\E}.\\
\end{split}
\]
Analogously, we show an upper bound on the stabilization term stemming from the discretization of the mixed formulation:
\[
\SE ( (\Ibold-\Piboldzp) \sigmaboldn, (\Ibold-\Piboldzp) \sigmaboldn) \le 4 \alpha^*  \Vert \k^{-\frac{1}{2}}(\sigmabold - \sigmaboldn) \Vert^2_{0,\E} + 2 \alpha^*  \Vert \k^{-\frac{1}{2}} (\sigmabold - \Piboldzp \sigmabold) \Vert^2_{0,\E}.
\]
Collecting the three estimates, we get
\begin{equation} \label{efficiency:global}
\begin{split}
\etahypE^2 	& = \Vert \k^{\frac{1}{2}} \nabla \Pinablap \utilden + \k^{-\frac{1}{2}} \Piboldzp \sigmaboldn \Vert^2_{0,\E} \\
			&\quad  + \StildeE( (I-\Pinablap) \utilden, (I-\Pinablap) \utilden) + \SE ( (\Ibold-\Piboldzp) \sigmaboldn, (\Ibold-\Piboldzp) \sigmaboldn) \\
			& \le 4(1 +  \alphatilde^*) \Vert \k^{\frac{1}{2}}  \nabla (\utilde - \utilden) \Vert^2_{0,\E}  + 2 (2+ \alphatilde^* ) \Vert \k^{\frac{1}{2}}  \nabla (\utilde - \Pinablap \utilden) \Vert^2_{0,\E} \\
			& \quad + 4(1 + \alpha^*) \Vert \k^{-\frac{1}{2}} (\sigmabold - \sigmaboldn) \Vert^2_{0,\E}  + 2 (2+ \alpha^*) \Vert \k^{-\frac{1}{2}} ( \sigmabold  - \Piboldzp \sigmaboldn) \Vert^2_{0,\E} .\\
\end{split}
\end{equation}
We have proven the following efficiency result.
\begin{thm} \label{theorem:bounds:efficiency}
Let the assumptions~(\textbf{G1}), (\textbf{G2}), (\textbf{K}), and (\textbf{D}) be valid.
Let~$\utilde$, and $u$ and~$\sigmabold$ be the solutions to~\eqref{primal:formulation} and~\eqref{mixed:formulation}, respectively,
and~$\utilden$, and $\un$ and~$\sigmaboldn$ be the solutions to~\eqref{VEM:primal} and~\eqref{VEM:mixed}, respectively.
The following upper bound on the local equilibrated error estimator~$\etahypE$ in terms of the local exact error and best local polynomial approximation terms is valid: for every~$\E \in \taun$,
\begin{equation} \label{global:B}
\begin{split}
\etahypE^2 	& \le 4(1 +  \alphatilde^*) \Vert \k^{\frac{1}{2}}  \nabla (\utilde - \utilden) \Vert^2_{0,\E}  + 2 (2+ \alphatilde^* ) \Vert \k^{\frac{1}{2}}  \nabla (\utilde - \Pinablap \utilden) \Vert^2_{0,\E} \\
			& \quad + 4(1 + \alpha^*) \Vert \k^{-\frac{1}{2}} (\sigmabold - \sigmaboldn) \Vert^2_{0,\E}  + 2 (2+ \alpha^*) \Vert \k^{-\frac{1}{2}} ( \sigmabold  - \Piboldzp \sigmaboldn) \Vert^2_{0,\E} .\\
\end{split}
\end{equation}
Bound~\eqref{global:B} is fully explicit in terms of~$\h$ and~$\p$ where the $\p$-dependence is also contained in the stabilization constants.
\end{thm}

\section{Numerical results} \label{section:nr}
In this section, we introduce an adaptive algorithm and an $\h\p$-refinement strategy in order to show the performance of the hypercircle method and to compare it with that of the residual based a posteriori approach.

We ought to compare the equilibrated error estimator~$\etahyp$ in~\eqref{global_EE} with the exact error of the method; see the left-hand side of~\eqref{global:A}.
However, since functions in virtual element spaces are not known in closed-form but only through their degrees of freedom, we cannot compute the exact errors.
Rather, we compare the equilibrated error estimator with the following computable approximation of the error:
given the solutions~$\sigmabold$, $u$, and $\sigmaboldn$, $\un$ to mixed problem~\eqref{mixed:formulation} and VEM~\eqref{VEM:mixed}, respectively,
define the approximate error for mixed VEM~\eqref{VEM:mixed} as
\begin{equation} \label{computable:error:mixed}
\left(  \Vert \k^{\frac{1}{2}} (\nabla u -  \nabla \Pinablap\un)    \Vert^2_{0,\Omega}  + \Vert \k^{-\frac{1}{2}} (\sigmabold - \Piboldzp \sigmaboldn ) \Vert_{0,\Omega}^2   \right)^{\frac{1}{2}}.
\end{equation}
The approximate error~\eqref{computable:error:mixed} converges with the same $\h$- and~$\p$- convergence rate as the exact error.
To see this, it suffices to use arguments analogous to those, e.g., in~\cite[Section~5]{hpVEMapos}.

Since we shall compare the performance of the equilibrated error estimator with that of the residual error estimator, we introduce a computable approximation of the error for the primal formulation as well:
given the solutions~$\utilde$ and~$\utilden$ to the primal problem~\eqref{primal:formulation} and VEM~\eqref{VEM:primal}, respectively,
define the approximate error for primal VEM~\eqref{VEM:primal} as
\begin{equation} \label{computable:error:primal}
\Vert \k^{\frac{1}{2}} (\nabla \utilde -  \nabla \Pinablap\utilden)    \Vert_{0,\Omega}  .
\end{equation}
The approximate error~\eqref{computable:error:primal} converges with the same $\h$- and~$\p$- convergence rate as the exact error~$\Vert \k^{\frac{1}{2}} (\nabla u - \nabla \un) \Vert_{0,\Omega}$; see, e.g., \cite[Section~5]{hpVEMcorner}.

As stabilizations for the primal and mixed formulation, we use those defined in~\eqref{practical:stab:primal} and~\eqref{practical:stab:mixed}, respectively.
We fix shifted and scaled monomials~\cite[equation (4.4)]{VEMvolley} as a polynomial basis in the definition of the internal degrees of freedom~\eqref{internal:moments:primal} for the primal formulation.
As for the choice of polynomial bases for internal moments~\eqref{gradientDOF:flux} and~\eqref{orthogonalDOF:flux} of the mixed formulation, we use those described in~\cite[Proposition~2.1]{VEMbricksMixed}.

\begin{remark} \label{remark:high-poly}
For high polynomial degrees, these choices of the polynomial bases are not the most effective. Rather, we ought to consider some sort of orthogonalization of the polynomial bases,
as proposed in~\cite{fetishVEM} for the primal formulation.
Since this has not been investigated for the mixed formulation so far, we postpone the investigation of the effects of the choice of the polynomial bases on the performance of the method to future works.
\eremk
\end{remark}

\subsection{Experiments on the theoretical and practical stabilizations}
Before proceeding with the numerical experiments for the adaptive algorithm, we show that the $\p$-version of the method converges optimally in terms of~$\p$.
Moreover, we compare numerically the theoretical and practical stabilizations presented above.
As for the primal formulation stabilizations~\eqref{stab:primal} and~\eqref{practical:stab:primal}, this has already been investigated in~\cite[Section~2.2]{fetishVEM}.
Indeed, the method works extremely similarly even with a high degree of accuracy.

As for the mixed formulation stabilizations~\eqref{stab:mixed} and~\eqref{practical:stab:mixed}, we provide here numerical evidence that they deliver similar results.
To this aim, consider a coarse mesh with some nonconvex elements as that depicted in Figure~\ref{figure:comparison-stab-mixed} (left).
With the $\p$-version of mixed VEM~\eqref{VEM:mixed}, we approximate  the exact smooth solution
\begin{equation} \label{smooth-solution}
u(x,y) = \sin(\pi x) \sin(\pi y) \quad \text{ on } \Omega := (0,1)^2.
\end{equation}
In Figure~\ref{figure:comparison-stab-mixed} (right), we show the decay of error~\eqref{computable:error:mixed} versus the polynomial degree.

\begin{figure}  [h]
\centering
\begin{center}
\begin{tikzpicture}[scale=1.1]
\draw[black, thick, -] (0,0) -- (4,0) -- (4,4) -- (0,4) -- (0,0);
\draw[black, thick, -] (2,0) -- (2,4); \draw[black, thick, -] (0,2) -- (4,2);
\draw[black, thick, -] (0,1) -- (1,1.5); \draw[black, thick, -] (1,1.5) -- (2,1); \draw[black, thick, -] (2,1) -- (3,1.5); \draw[black, thick, -] (3,1.5) -- (4,1);
\draw[black, thick, -] (0,3) -- (1,3.5); \draw[black, thick, -] (1,3.5) -- (2,3); \draw[black, thick, -] (2,3) -- (3,3.5); \draw[black, thick, -] (3,3.5) -- (4,3);
 \draw[fill=white, draw=white] (0,-.45) circle (3pt);
\end{tikzpicture}
\quad \quad \quad
\includegraphics [angle=0, width=0.48\textwidth]{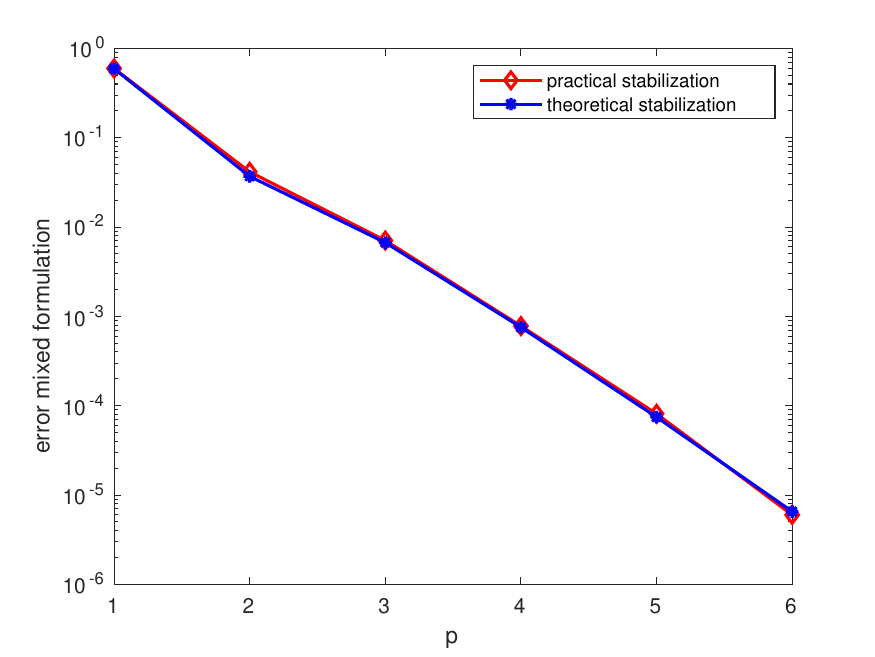}
\end{center}
\caption{(\emph{Left panel:}) a coarse mesh consisting of~$8$ elements, $4$ of which being nonconvex.
(\emph{Right panel:}) the $\p$-version of mixed VEM~\eqref{VEM:mixed} applied to the exact solution in~\eqref{smooth-solution}.
We employ theoretical~\eqref{stab:mixed} and practical~\eqref{practical:stab:mixed} stabilizations.}
\label{figure:comparison-stab-mixed}
\end{figure}
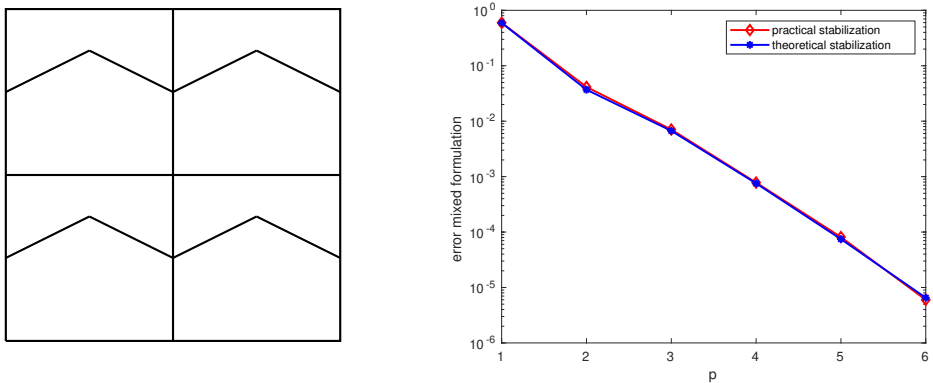
At the practical level, the two stabilizations deliver similar, indeed practically identical, results. Since the practical one is easier to implement and faster to run, we henceforth stick to it.
\medskip

\subsection{The test cases}
We test the performance of the virtual element method and the equilibrated error estimator on the following test cases.
\paragraph*{Test case 1.}
The first test case is defined on the L-shaped domain
\[
\Omega_1 = (-1,1)^2 \setminus \big \{ [0,1) \times(-1,0] \big\} .
\]
In polar coordinates centred at~$(0,0)$, the solution to this problem reads
\begin{equation} \label{u1}
 u_1(r,\theta) = r^{\frac{2}{3}} \sin \left (\frac{2}{3} \theta \right).
\end{equation}
The primal formulation of the problem we are interested in is such that we have: zero Dirichlet boundary conditions on the edges generating the re-entrant corner; suitable Neumann boundary conditions on all the other edges;
$\k=1$; zero right-hand side, since~$u_1$ is harmonic.

\paragraph*{Test case 2.}
The second test case is defined on the slit domain
\[
\Omega_2 := (-1,1)^2 \setminus \big\{ [0,1) \times \{0\}  \big \}.
\]
In polar coordinates centred at~$(0,0)$, the solution to this problem reads
\begin{equation} \label{u2}
u_2(r,\theta) = r^{\frac{1}{4}} \sin \left(\frac{1}{4} \theta \right).
\end{equation}
The primal formulation of the problem we are interested in is such that we have: suitable Neumann boundary conditions on the bottom edge of the slit;
Dirichlet boundary conditions on all the other edges; $\k=1$; zero right-hand side, since~$u_2$ is harmonic.
Note that the slit consists of two boundary edges: the bottom and the upper part of the slit.
\medskip

We depict the solutions to the two test cases in Figure~\ref{figure:test-cases}.
\medskip

\begin{figure}[h]
\centering
\includegraphics [angle=0, width=0.45\textwidth]{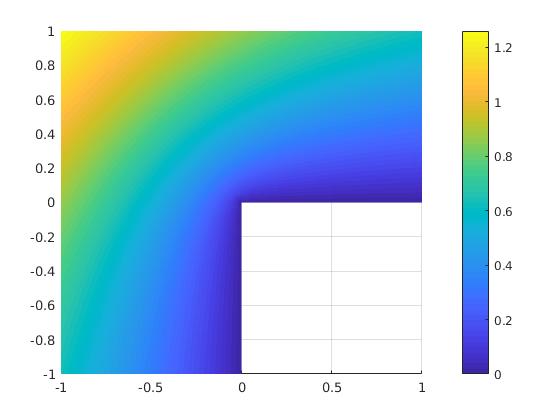}
\includegraphics [angle=0, width=0.45\textwidth]{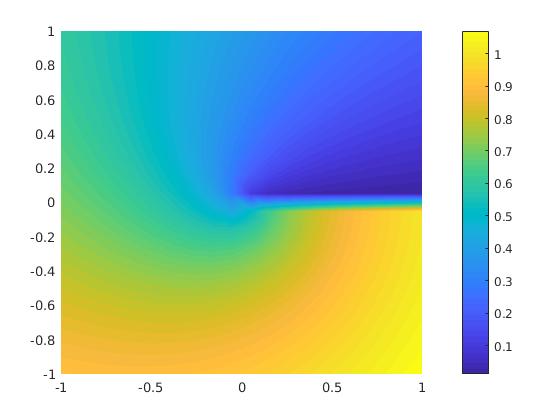}
\caption{Test cases. (\textit{Left panel:}) $u_1$ defined in~\eqref{u1}. (\textit{Right panel:}) $u_2$ defined in~\eqref{u2}.}
\label{figure:test-cases}
\end{figure}

The remainder of the section is organized as follows.
In Section~\ref{subsection:residual}, we recall the residual virtual element error estimator from~\cite{hpVEMapos}.
Section~\ref{subsection:efficiency} is devoted to analyze the behaviour of the effectivity index for the $\p$-version of VEM, when using the residual and the equilibrated error estimators.
The adaptive algorithm and the $\h\p$-refinements are described in Section~\ref{subsection:adaptive}, whereas the performance of the $\h\p$-adaptive algorithm is analyzed in Section~\ref{subsection:hp_hyper}.

\subsection{The residual error estimator} \label{subsection:residual}
In this section, we recall the residual error estimator derived in~\cite[Section~4]{hpVEMapos} for the $\h\p$-version of the virtual element method and its properties.
We assume that the diffusion coefficient~$\k$ is equal to~$1$, since this was the instance considered in~\cite{hpVEMapos}.

Given a mesh~$\taun$ and a distribution of degrees of accuracy as in Remarks~\ref{remark:hp:primal} and~\ref{remark:hp:mixed},
introduce the following local residual error estimators: given the solution~$\utilden$ to~\eqref{VEM:primal}, for all~$\E \in \taun$,
\[
\begin{split}
\etaResE^2 		&:=  \frac{\hE^2}{\p^2} \Vert \Delta \Pinablap \utilden + \Pitildez \f \Vert^2_{0,\E} + \frac{1}{2} \sum_{\e \in \EE,\, \e \not \subset \GammaD} \frac{\hE}{\p} \left\Vert \left\llbracket \n_\e \cdot \nabla \Pinablap \utilden \right\rrbracket    \right\Vert^2_{0,\e}\\
				& \quad + \SE( (I-\Pinablap) \utilden, (I-\Pinablap) \utilden).\\
\end{split}
\]
The global residual error estimator is defined as
\begin{equation} \label{residual:error_estimator}
\etaRes ^2 :=  \sum_{\E \in \taun} \etaResE^2.
\end{equation}
In~\cite[Theorem~1]{hpVEMapos}, the authors proved lower and upper bounds of the residual error estimator in terms of the error of the primal formulation.
Although such bounds are optimal in terms of the mesh size, they are suboptimal in terms of the degree of accuracy of the method. This resembles what happens in the finite element method framework;
see~\cite[Theorem 3.6]{MelenkWohlmuth_hpFEMaposteriori}.
The suboptimality is due to the use of polynomial inverse estimates when proving the efficiency.

\subsection{Effectivity index: residual versus equilibrated error estimators} \label{subsection:efficiency}
The aim of the present section is to investigate the behaviour of the effectivity indices of the residual~\eqref{residual:error_estimator} and equilibrated~\eqref{global_EE} error estimators.
In particular, we demonstrate the numerical $\p$-robustness of the latter.

We define the effectivity index of the two error estimators as follows:
\begin{equation} \label{efficiency:two_EE}
\begin{split}
& \Ihyp^2  := \frac{\etahyp^2}{ \Vert \k^{\frac{1}{2}} (\nabla u - \nabla \Pinablap \utilden) \Vert^2_{0,\Omega}}, \qquad
\Ires^2 := \frac{\etaRes^2}{\Vert \k^{\frac{1}{2}} (\nabla \utilde - \nabla \Pinablap \utilden) \Vert_{0,\Omega}^2}.
\end{split}
\end{equation}
We run the $\p$-version of the method with exact solution~$u_1$ defined in~\eqref{u1}.
As an underlying mesh, we fix a uniform Cartesian mesh with~$12$ elements.
The behaviour of efficiency indices~\eqref{efficiency:two_EE} is depicted in Figure~\ref{figure:efficiency_index}.

\begin{figure}  [h]
\centering
\includegraphics [angle=0, width=0.5\textwidth]{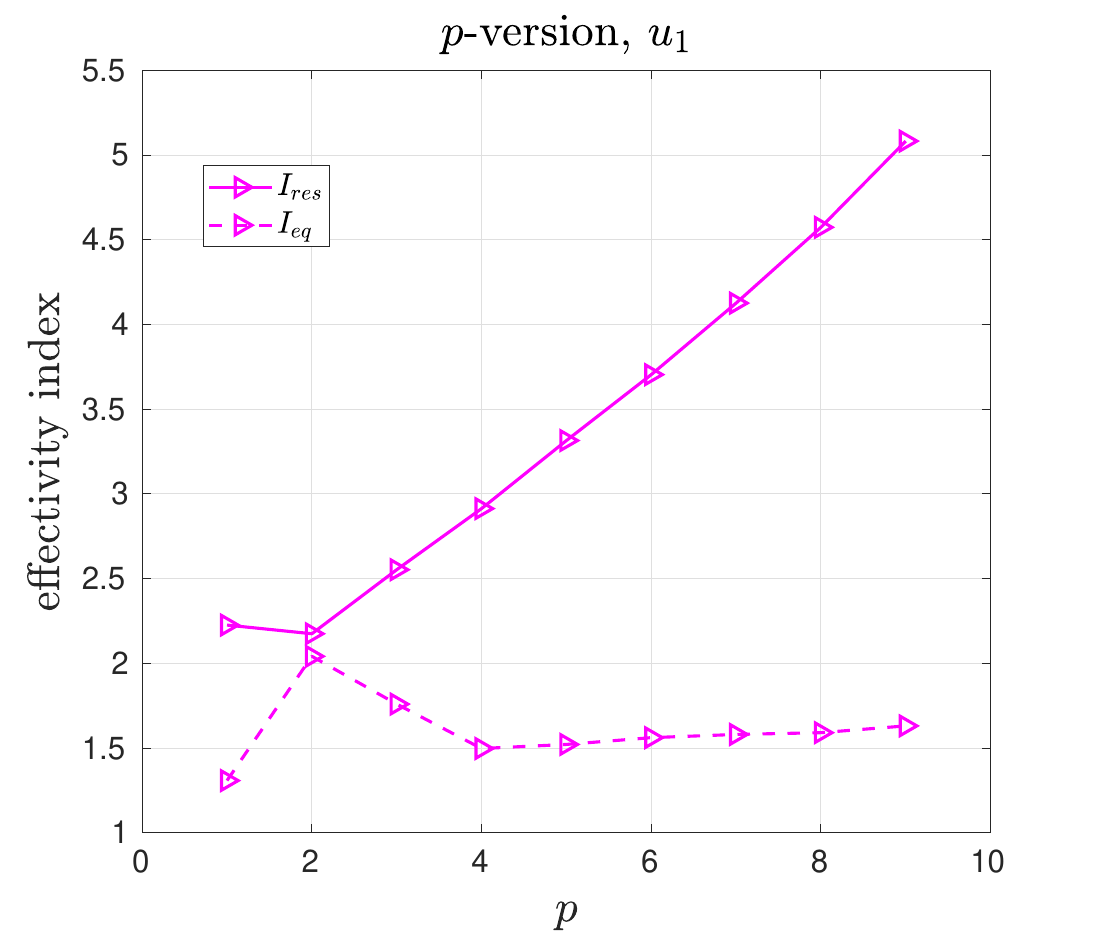}
\caption{Effectivity indices~$\Ires$ and~$\Ihyp$~\eqref{efficiency:two_EE} for the $\p$-version of the method,
using residual~$\etaRes$~\eqref{residual:error_estimator} and equilibrated~$\etahyp$~\eqref{global_EE} error estimators.
We consider the exact solution~$u_1$ in~\eqref{u1} and use a uniform Cartesian mesh consisting of~$12$ elements.}
\label{figure:efficiency_index}
\end{figure}

From Figure~\ref{figure:efficiency_index}, we observe that the effectivity index for the hypercircle method seems to be independent of~$\p$, differently from that of the residual error estimator.
On the one hand, this is in partial agreement with bounds~\eqref{reliability:global} and~\eqref{efficiency:global}. Here, no polynomial inverse estimates have been used.
On the other hand, the bounds depend on the stability constants of the method. As shown in~\eqref{bounds:stab_primal} and~\eqref{bounds:stab_mixed}, the stability constants might depend on~$\p$.
Notwithstanding, it seems that such bounds are crude, and the stability constants do not play a role in terms of~$\p$.
We performed analogous experiments on the test case~$u_2$ and obtained comparable results, which we omit for the sake of brevity.
Eventually, note that the effectivity index for the equilibrated error estimator remains close to~$1.5$.

\subsection{The adaptive algorithm and $\h\p$-adaptive mesh refinements} \label{subsection:adaptive}
In this section, we recall the structure of an adaptive algorithm, the meaning of $\h$- and $\p$-refinement, and how to choose between~$\h$- and $\p$-refinements.
The standard structure of an adaptive algorithm is
\begin{center}
\textbf{SOLVE} $\quad \longrightarrow \quad$ \textbf{ESTIMATE} $\quad \longrightarrow \quad$ \textbf{MARK} $\quad \longrightarrow \quad$ \textbf{REFINE}.
\end{center}
The remainder of this section is devoted to address the \emph{marking} and \emph{refining} steps. The latter consists in deciding whether to refine a marked element either in~$\h$ or in~$\p$.
Refining in~$\p$, the local space on an element~$\E$ means that local degree of accuracy~$\pE$ on~$\E$ is increased by one.
The design of the global space and its degrees of freedom is performed accordingly to Remarks~\ref{remark:hp:primal} and~\ref{remark:hp:mixed}.

We describe the $\h$-refinement in more details. Firstly, we anticipate that we shall employ Cartesian and triangular meshes, only.
This might seem idiosyncratic, as we claimed that we want an adaptive method working on general meshes.
However, the flexibility in employing polygons is exploited when refining squares or triangles and creating hanging nodes.
In the refining procedure, we define a geometric square as a geometrical entity with four straight edges. For instance, a polygon with five vertices and having two adjacent edges on the same line is a geometric square.
Analogously, we define a geometric triangle as a geometrical entity with three straight edges.
A geometric square is refined into four smaller geometric squares by connecting its centroid to the midpoints of the four straight edges; see Figure~\ref{figure:square:refinement} (left).
Instead, a geometric triangle is refined into four smaller geometric triangles as in Figure~\ref{figure:square:refinement} (right).

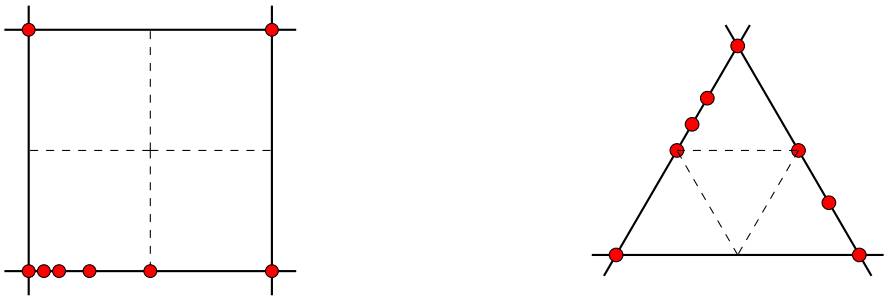
\begin{figure}[h]
\centering
\begin{minipage}{0.32\textwidth}
\begin{center}
\begin{tikzpicture}[scale=0.8]
\draw[black, thick, -] (-0.4,0) -- (4.4,0); \draw[black, thick, -] (0, -0.4) -- (0, 4.4); \draw[black, thick, -] (4,-.4) -- (4,4.4); \draw[black, thick, -] (-0.4,4) -- (4.4,4);
\draw[black, dashed] (2,2) -- (2,0);  \draw[black, dashed] (2,2) -- (2,4);  \draw[black, dashed] (2,2) -- (0,2);  \draw[black, dashed] (2,2) -- (4,2); 
\draw[fill=red] (0,0) circle (3pt); \draw[fill=red] (4,0) circle (3pt); \draw[fill=red] (0,4) circle (3pt); \draw[fill=red] (4,4) circle (3pt); \draw[fill=red] (2,0) circle (3pt); \draw[fill=red] (1,0) circle (3pt); \draw[fill=red] (0.5,0) circle (3pt); \draw[fill=red] (0.25,0) circle (3pt);
\end{tikzpicture}
\end{center}
\end{minipage}
\quad\quad\quad\quad \quad\quad\quad\quad
\begin{minipage}{0.32\textwidth}
\begin{center}
\begin{tikzpicture}[scale=3.2]
\draw[black, thick, -] (-.1,0) -- (1.1,0);  \draw[black, thick, -](1+.1*.5,-0.1*sqrt 3/2) -- (1/2-.1*.5, sqrt 3/2+.1*sqrt 3/2); \draw[black, thick, -] (1/2+.1*.5, sqrt 3/2 +.1*sqrt 3/2) -- (-.1*.5, -.1*sqrt 3/2);
\draw[fill=red] (0,0) circle (0.8pt); \draw[fill=red] (1,0) circle (.8pt); \draw[fill=red] (1/2, sqrt 3/2) circle (.8pt);
\draw[fill=red] (1/4, sqrt 3 /4) circle (0.8pt); \draw[fill=red] (3/4, sqrt 3/4) circle (0.8pt);
\draw[fill=red] (3/8, 3/8*sqrt 3) circle (0.8pt); \draw[fill=red] (5/16, 5/16*sqrt 3) circle (0.8pt); \draw[fill=red] (7/8, sqrt 3/ 8) circle (0.8pt);
\draw[black, dashed] (1/4, sqrt 3 /4) -- (3/4, sqrt 3/4); \draw[black, dashed] (3/4, sqrt 3/4) -- (1/2, 0); \draw[black, dashed] (1/2, 0) -- (1/4, sqrt 3 /4);
\end{tikzpicture}
\end{center}
\end{minipage}
\caption{(\emph{Left panel:}) Refining a geometric square with~$8$ vertices, $8$ edges, and~$4$ straight edges.
The refinement is performed by connecting the centroid to the midpoints of the four straight edges.
(\emph{Right panel:}) Refining a geometric triangle with~$8$ vertices, $8$ edges, and~$3$ straight edges.
The refinement is performed by subdividing the geometric triangle into~$4$ geometric subtriangles with vertices given by the original vertices and the midpoints of the three straight edges.
In red, we depict the vertices of the polygon under consideration.}
\label{figure:square:refinement}
\end{figure}
Observe that $\h$-refinements of Cartesian and triangular meshes lead to meshes always consisting of geometric squares and triangles, respectively.
Note that a refinement in presence of a hanging node in the proper place for the $\h$-refinement is not creating an additional node.

We are left with the description of the marking strategy. Firstly, we describe how to choose the elements to mark. Secondly, we set a way to decide whether to mark for $\h$- or $\p$-refinement.
Marking elements is a rather standard procedure. We mark for refinement all the elements~$\E \in \taun$ such that, given a positive parameter~$\sigma \in (0,1)$,
\[
\etahypE \ge \sigma \, \etahypbar :=\sigma \frac{\etahyp}{\text{card}(\taun)}.
\]
As for the $\h$- or $\p$-marking, we could use several strategies; see for instance the survey paper~\cite{mitchell2011survey}.
In words, the idea behind this choice resides in refining the mesh on the marked elements where the solution is expected to be singular.
An increase of the degree of accuracy is performed on the marked elements where the solution is expected to be smooth.
Amongst the various techniques available in the literature, we follow the approach of Melenk and Wohlmuth, see~\cite[Section~4]{MelenkWohlmuth_hpFEMaposteriori},
which is based on comparing the actual equilibrated error estimator with a predicted one.

In the forthcoming numerical experiments, we set~$\sigma=1$, $\lambda=0.2$, $\gamma_\h=1$, $\gamma_\p=1$, and~$\gamma_n=1$ in~\cite[Algorithm~4.4]{MelenkWohlmuth_hpFEMaposteriori}.

\subsection{The $\h$- and $\h\p$-adaptive algorithm} \label{subsection:hp_hyper}
In this section, we present several numerical experiments on $\h$- and $\h\p$-adaptivity employing the equilibrated error estimator in~\eqref{global_EE}
and Melenk-Wohlmuth's refining strategy~\cite[Section~4]{MelenkWohlmuth_hpFEMaposteriori}.
We consider the two different test cases in~\eqref{u1}--\eqref{u2}, and compare the performance of the $\h$-adaptive algorithm with~$\p=1$, $2$, and~$3$, and the $\h\p$-version of the method.

In Figures~\ref{figure:hpVSh:testcase1}--\ref{figure:hpVSh:testcase2},
we depict the performance of the $\h$- (with~$\p=1$, $2$, and~$3$) and $\p$-adaptive algorithms for all the test cases introduced in~\eqref{u1}--\eqref{u2}.
We start with a coarse Cartesian mesh on the left and a coarse mesh made of structured triangles on the right.

We plot error~\eqref{computable:error:mixed} versus the cubic root of the number of degrees of freedom:
we expect exponential convergence of the error in terms of the cubic root of the number of degrees of freedom, when employing an optimal $\h\p$-mesh.
To see this, one has to  combine the techniques in~\cite{SchoetzauWihler_hpMixed, hpVEMcorner}.

\begin{figure}  [h]
\centering
\includegraphics [angle=0, width=0.48\textwidth]{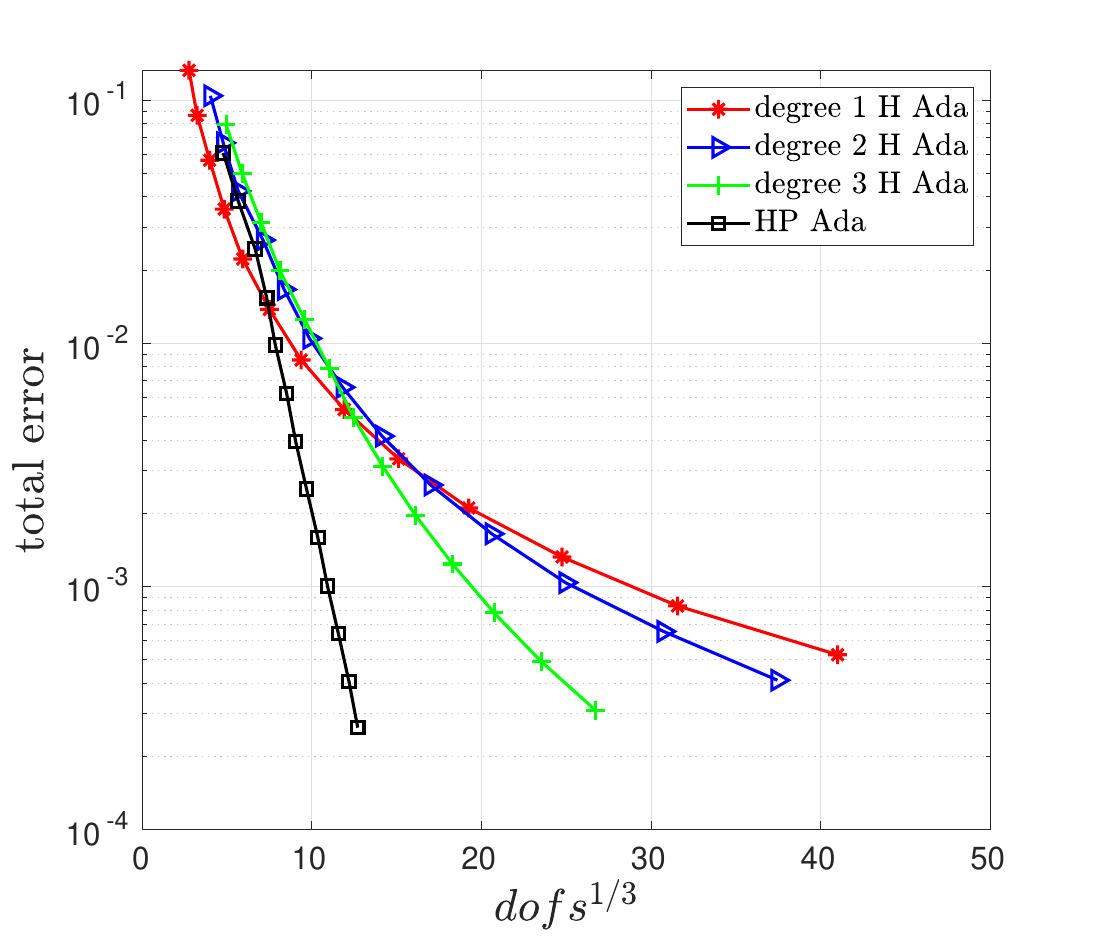}
\includegraphics [angle=0, width=0.48\textwidth]{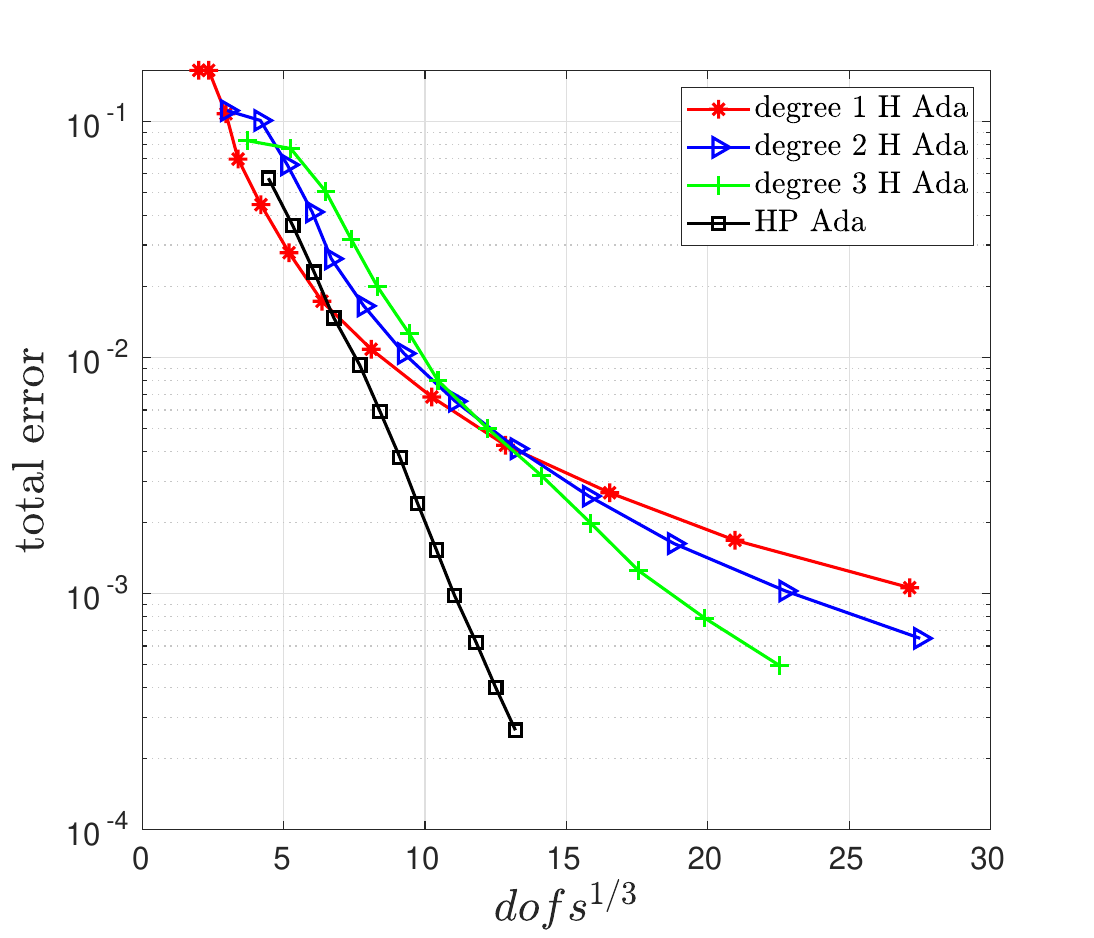}
\caption{$\h$- (with $\p=1$, $2$, and~$3$) versus $\h\p$-adaptive algorithm. The solution is~$u_1$ defined in~\eqref{u1}.
The starting mesh is (\textit{left panel:}) a coarse Cartesian mesh and (\textit{right panel:}) a coarse mesh of structured triangles.}
\label{figure:hpVSh:testcase1}
\end{figure}

\begin{figure}  [h]
\centering
\includegraphics [angle=0, width=0.48\textwidth]{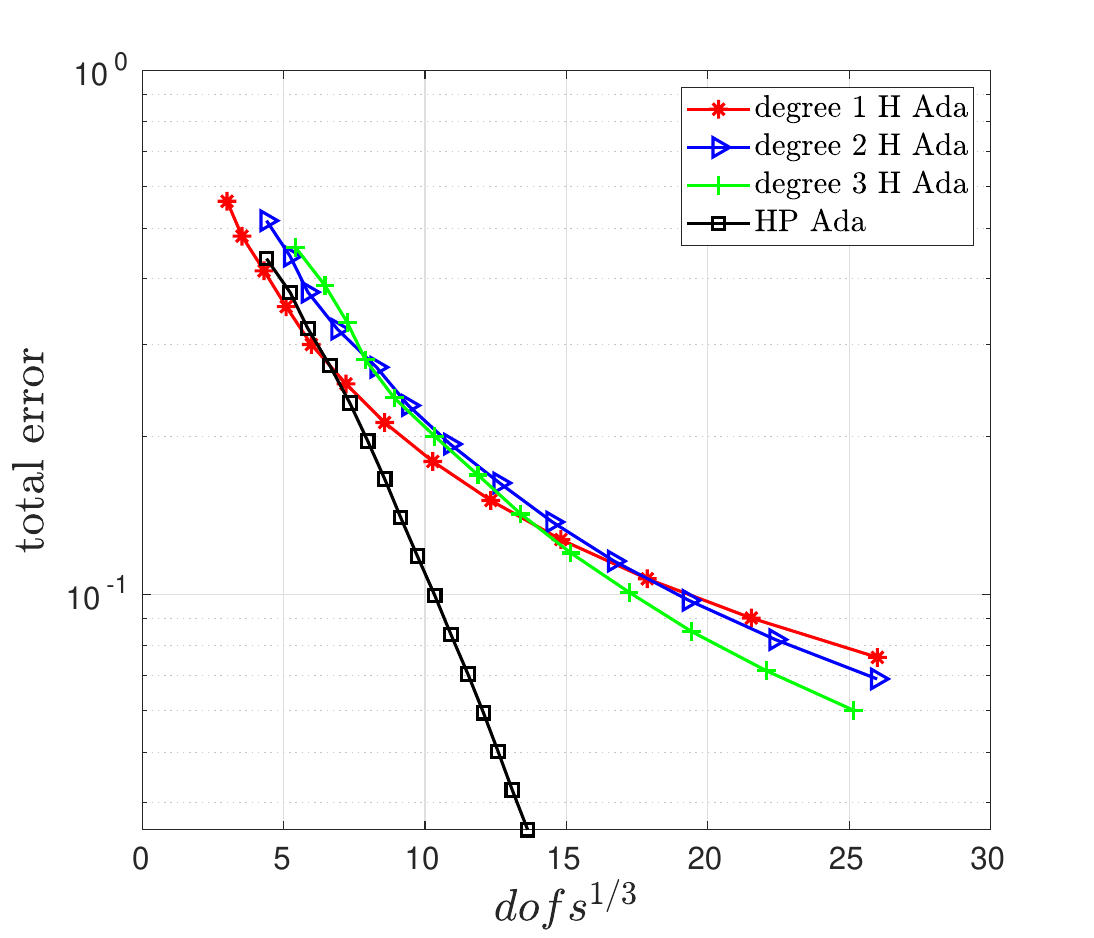}
\includegraphics [angle=0, width=0.48\textwidth]{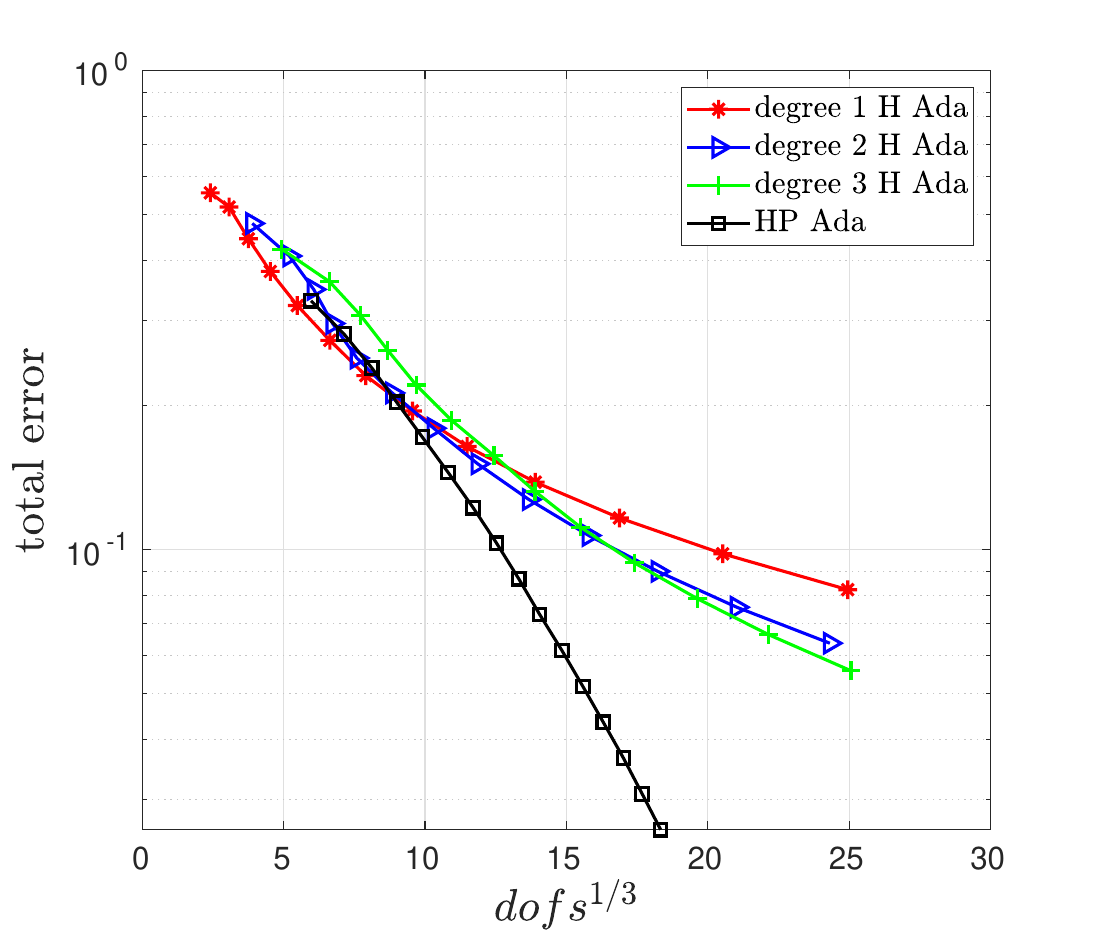}
\caption{$\h$- (with $\p=1$, $2$, and~$3$) versus $\h\p$-adaptive algorithm. The solution is~$u_2$ defined in~\eqref{u2}.
The starting mesh is (\textit{left panel:}) a coarse Cartesian mesh and (\textit{right panel:}) a coarse mesh of structured triangles.}
\label{figure:hpVSh:testcase2}
\end{figure}

From Figures~\ref{figure:hpVSh:testcase1}--\ref{figure:hpVSh:testcase2}, we observe the exponential decay of the error in terms of the cubic root of the number of degrees of freedom
for $\h\p$-adaptive mesh refinements, and algebraic convergence for $\h$-adaptive refinements.
The $\h\p$-adaptive version leads to smaller errors with fewer degrees of freedom compared to the $\h$-adaptive version.

Finally, we exhibit the $\h\p$-meshes after~$5$ and~$14$ refinements of the adaptive algorithm for the test case 1; see Figure~\ref{figure:hp-meshes}.
\begin{figure}  [h]
\centering
\includegraphics [angle=0, width=0.2\textwidth]{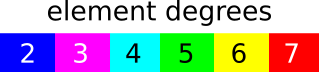}\\
\medskip\medskip
\includegraphics [angle=0, width=0.4\textwidth]{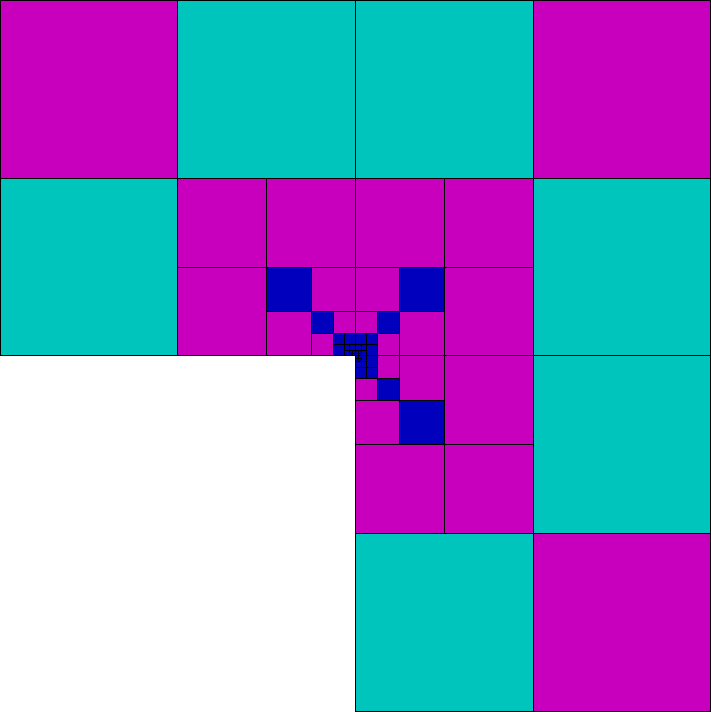}
\quad\quad\quad
\includegraphics [angle=0, width=0.4\textwidth]{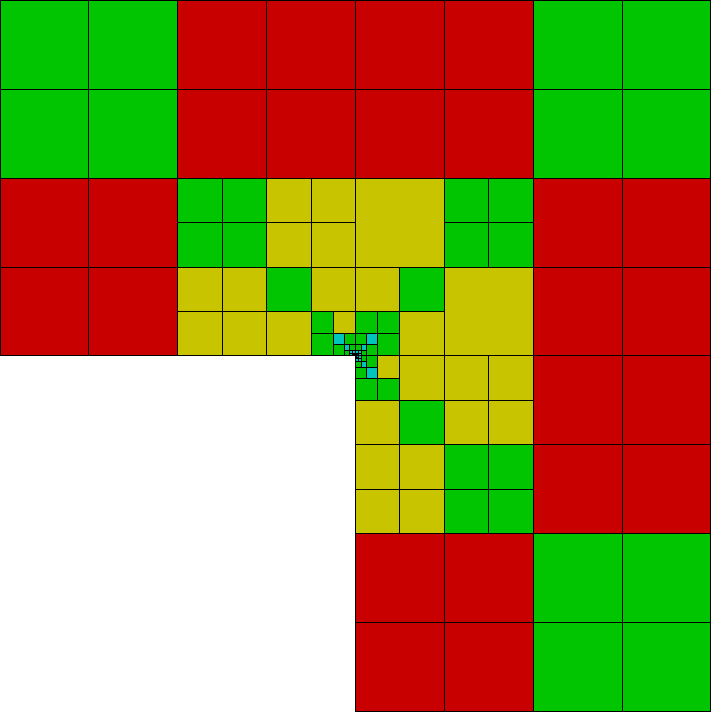}
\caption{$\h\p$-adaptive algorithm. The solution is~$u_1$ defined in~\eqref{u1}.
We exhibit the $\h\p$-meshes after~$5$ (\emph{left panel:}) and~$14$ (\emph{right panel:}) refinements of the adaptive algorithm.}
\label{figure:hp-meshes}
\end{figure}

\section{Local flux reconstruction in VEM: a first investigation} \label{section:localization-mixed}
In Section~\ref{section:apos}, we proved lower and upper bounds of the exact error in terms of an equilibrated error estimator where the dependence on the distribution of the degrees of accuracy is isolated within the stability constants.
This is a major improvement compared to the results achieved in the residual error estimator setting; see~\cite[Theorem~1]{hpVEMapos}.
Nonetheless, the linear system associated with the mixed method~\eqref{VEM:mixed} has approximately three times the number of unknowns of the linear system associated with the primal formulation~\eqref{VEM:primal}.
This downside can be overcome via the localization of the mixed VEM.
This has been already investigated in several works within the continuous and discontinuous finite element framework; see~\cite{braess2009equilibrated} and~\cite{ern2015polynomial}, respectively, and the references therein.

In words, the localization technique works as follows.
We construct an error estimator, which can be computed with the degrees of freedom of the solution to primal VEM~\eqref{VEM:primal} \emph{and} those of a cheap to compute numerical flux.
Whilst in Section~\ref{section:apos}, this function was given by the solution to a local mixed VEM~\eqref{VEM:mixed},
here, it is provided by the combination of solutions to local mixed VEM, which are cheap to solve and can be parallelized.

The aim of this section is to provide an initial study towards the local flux reconstruction in VEM.
We shall be able to prove that the error estimator is reliable and satisfies a condition on equilibration of fluxes.
However, we shall not prove the efficiency. This is also reflected in the numerical results of Section~\ref{section:nr:local}.

\paragraph*{Notation and assumptions.}
To simplify the forthcoming analysis, we henceforth assume
\[
\k=1, \quad \GammaD = \partial \Omega,\quad \GammaN = \emptyset.
\]
We define the polygonal patch around a vertex~$\nu \in \Nun$ and the elements of the mesh~$\taun$ belonging to such patch as
\[
\omeganu := \bigcup \left \{  \E \in \taun \mid \nu \in \NuE    \right\}, \quad \quad \taunomeganu = \{ \E \in \taun \mid \E \subset \omeganu   \}.
\]
Define the two following element spaces over the patch~$\omeganu$:
\[
\begin{split}
& \Vn(\omeganu) :=   \mathcal S^{\p-1, -1} (\omeganu, \taunomeganu) , \\
& \Sigmaboldn(\omeganu) :=  \{  \tauboldn \in H(\div, \omeganu ) \mid \tauboldn{}_{|\E} \in \Sigmaboldn(\E) \quad \forall \E \in \taunomeganu,  \; \llbracket \tauboldn \rrbracket _{\e} =0\quad \forall \e \in  \EnI  \}.\\
\end{split}
\]
We also define
\[
\Sigmaboldnz (\omeganu) := \left\{ \tauboldn  \in \Sigmaboldn (\omeganu)  \mid \n \cdot \tauboldn = 0 \text{ on } \partial \omeganu   \right\}, 
\quad \quad
\Vnstar(\omeganu) :=  \left\{ \vn \in \Vn(\omeganu) \middle| \int_{\omeganu} \vn = 0  \right\}.
\]
For each patch~$\omeganu$, we introduce localized discrete bilinear form
\[
\anomeganu (\sigmaboldnnu, \tauboldnnu) := \sum_{\E \in \taunomeganu} \anE(\sigmaboldnnu, \tauboldnnu) \quad \quad \forall \sigmaboldnnu,\, \tauboldnnu \in \Sigmaboldn(\omeganu).
\]

\paragraph*{Construction of a virtual element partition of unity.}
We construct a virtual element partition of unity~$\{\phitildenu\}_{\nu \in \Nun}$ associated with the mesh~$\taun$ by suitably fixing the degrees of freedom of each~$\phitildenu$.
To each vertex~$\nu \in \Vcaln$, we associate a function~$\phitildenu \in \Vtilden$ defined through its degrees of freedom as follows.
It is equal to~$1$ at~$\nu$, annihilates at all the other vertices, and is affine on the skeleton of the mesh.

For all~$\E \in \taun$ with~$\nu \in \mathcal V^\E$, we proceed as follows.
Let~$\NE$ be the number of vertices of~$\E$. Then, we set
\[
\frac{1}{\vert \E \vert} \int_\E \phitildenu \malphaE = \frac{1}{\NE\vert \E \vert} \int_\E \malphaE \quad\quad \forall \vert \boldalpha \vert =0, \dots, \p-2.
\]
On all the elements~$\E \in \taun$ such that~$\nu$ is not a vertex of~$\E$, $\phitildenu$ is extended by~$0$.

Indeed, $\{\phitildenu\}_{\nu \in \Vcaln}$ is a partition of unity. To see this, define~$\phitilde = \sum_{\nu \in \Vcaln} \phitildenu$. The restriction of~$\phitilde$ on the skeleton of the mesh is equal to~$1$.
Besides, the moments against piecewise polynomials up to degree~$\p-2$ are equal to the moments of the constant function~$1$.
The unisolvence of the degrees of freedom of the primal virtual element space entails the assertion; see Section~\ref{section:VEM:primal}.

Such a construction differs from the others in the VEM literature; see, e.g., \cite{Helmholtz-VEM, XVEM_2019}.
In these references, the partition of unity functions are defined via a harmonic lifting in each element, which implies that
the internal moments are not known, although they are later necessary for the computation of the various polynomial projectors.

\paragraph*{Local flux reconstruction for interior patches.}
We are in the position of defining the local VEM in mixed form.
For all~$\nu \in \NunI$, set
\begin{equation} \label{cE}
\cE = \frac{1}{\vert \E \vert} \StildeE( (I-\Pinablap) \phitildenu, (I-\Pinablap) \utilden),
\end{equation}
and consider the problem
\begin{equation} \label{local:mixed}
\begin{cases}
\text{find } (\sigmaboldnnu, \rnnu) \in \Sigmaboldnz(\omeganu) \times \Vnstar(\omeganu) \text{ such that, for all } \tauboldnnu \in \Sigmaboldnz (\omeganu) \text{ and } \qnnu \in \Vnstar (\omeganu), \\
\anomeganu (\sigmaboldnnu, \tauboldnnu) - (\div (\tauboldnnu), \rnnu)_{0,\omeganu} = - \sum_{\E \in \taunomeganu}  (\Pitildez \phitildenu \, \nabla \Pinablap \utilden, \Piboldzp \tauboldnnu)_{0, \E} \\
- (\div \sigmaboldnnu, \qnnu)_{0,\omeganu} = - \sum_{\E \in \taunomeganu} (\Pitildez \phitildenu \, \f - \nabla \Pinablap \phitildenu \cdot \nabla \Pinablap\utilden - \cE, \qnnu)_{0, \E} . \\
\end{cases}
\end{equation}

\begin{remark}
The second equation in~\eqref{local:mixed} represents the condition on equilibration of fluxes. This will become apparent in the sense of Lemma~\ref{lemma:p-version:div-property} below.
On the other hand, the first equation in~\eqref{local:mixed} is the residual equation. Roughly speaking, it is related to the fact that
\[
\nabla \rnnu \approx \nabla \utilden + \sigmaboldnnu,
\]
its weak formulation, and the VEM setting.
The right-hand side of the first equation in~\eqref{local:mixed} does not play a role in the proof of the reliability. However, it mimics the approach of, e.g., \cite{ern2015polynomial},
where it plays an important role when proving the efficiency.
\eremk
\end{remark}

Thanks to definition~\eqref{cE}, the second equation in~\eqref{local:mixed} is valid also for functions without zero average.
To see this, use that~$\phitildenu$ is equal to~$0$ outside the patch~$\omeganu$, and pick $\qnnu = 1$ in the second equation of~\eqref{local:mixed}, to get
\begin{equation} \label{local:compatibility}
\begin{split}
0 								& = \int_{\partial \omeganu} \n \cdot \sigmaboldnnu = \int_{\omeganu} 	\div \sigmaboldnnu	 \\
								& = \sum_{\E \in \taunomeganu}  \left(\int_{\E} \left[ \Pitildez \phitildenu\, \f - \nabla \Pinablap \phitildenu \cdot \nabla \Pinablap \utilden \right] - \frac{1}{\vert \E \vert} \StildeE( (I-\Pinablap) \phitildenu, (I-\Pinablap) \utilden) (1,1)_{0, \E} \right)\\
								& = (\Pitildez \phitildenu, \f)_{0,\omeganu} - \sum_{\E \in \taunomeganu} \left( ( \nabla \Pinablap \phitildenu, \nabla \Pinablap \utilden)_{0,\E} + \StildeE( (I-\Pinablap) \phitildenu, (I-\Pinablap) \utilden) \right)\\
								& = (\Pitildez \phitildenu, \f)_{0, \Omega} - \sum_{\E \in \taun} \left( (\nabla \Pinablap \phitildenu, \nabla \Pinablap \utilden)_{0,\Omega} {+} \StildeE( (I-\Pinablap) \phitildenu, (I-\Pinablap) \utilden) \right) \overset{\eqref{VEM:primal}}{=} 0. \\
\end{split}
\end{equation}
For all~$\nu \in \Nun$, problem~\eqref{local:mixed} is well-posed.
To see this, it suffices to use the compatibility condition~\eqref{local:compatibility} together with Theorem~\ref{theorem:inf-sup}, Remark~\ref{remark:inf_sup:BFM}, and the Babu\v ska-Brezzi theory.
The inf-sup condition can be proved as in Theorem~\ref{theorem:inf-sup}.

\paragraph*{Local flux reconstruction for boundary patches.}
We define the local mixed VEMs for boundary vertices in a slightly different fashion.
Given~$\nu \in \NunB$, set
\[
\Sigmaboldnz^{\GammaD}(\omeganu) := \{ \tauboldn  \in \Sigmaboldn (\omeganu)  \mid \n_{|\e} \cdot \tauboldn = 0 \text{ for all } \e \in \mathcal E^{\omeganu}_n,\, \e \not \subset \GammaD    \}.
\]
We consider local VEMs of the following form:
\begin{equation} \label{local-problem:DN}
\begin{cases}
\text{find } (\sigmaboldnnu, \rnnu) \in \Sigmaboldnz^{\GammaD}(\omeganu) \times \Vn(\omeganu) \text{ such that, for all } \tauboldnnu \in \Sigmaboldnz^{\GammaD} (\omeganu) \text{ and } \qnnu \in \Vn(\omeganu), \\
\anomeganu (\sigmaboldnnu, \tauboldnnu) - (\div (\tauboldnnu), \rnnu)_{0,\omeganu} = - \sum_{\E \in \taunomeganu} (\Pitildez \phitildenu \, \nabla \Pinablap \utilden, \Piboldzp \tauboldnnu)_{0,\E} \\
- (\div \sigmaboldnnu, \qnnu)_{0,\omeganu} = - \sum_{\E \in \taunomeganu}  (\Pitildez \phitildenu \, \f - \nabla \Pinablap \phitildenu \cdot \nabla \Pinablap \utilden - \cE, \qnnu)_{0,\E}. \\
\end{cases}
\end{equation}
To prove the well-posedness of problem~\eqref{local-problem:DN}, it suffices to use arguments similar to those employed in the proof of Theorem~\ref{theorem:inf-sup}.

Next, define
\begin{equation} \label{global:sigma}
\sigmaboldn =: \sum_{\nu \in \Vcaln}  \sigmaboldnnu,
\end{equation}
and note that~$\sigmaboldn$ belongs to~$\Sigmaboldn$ by construction.

\begin{lem} \label{lemma:p-version:div-property}
Let~$\sigmaboldn$ be defined as in~\eqref{global:sigma}. For all~$\E \in \taun$, for all~$\vn$ in~$\Vn(\E)$, i.e., for all~$\vn$ in~$\mathbb P_{\p-1}(\E)$, the following identity is valid:
\begin{equation} \label{higher-moments:divergence}
\int_\E \div \sigmaboldn \vn= \int_\E \f \vn .
\end{equation}
\end{lem}
\begin{proof}
Fix~$\E \in \taun$.
We have that
\[
\sigmaboldn{}_{|\E} = \sum_{\nu \in \NuE} \sigmaboldnnu.
\]
Let~$\nu \in \NuE$.
Recall that~\eqref{local:compatibility} entails that the second equation in~\eqref{local:mixed} is valid for test functions without zero average.
Pick~$\qnnu$ equal to a polynomial~$\vn$ in~$\E$ and zero elsewhere in the second equation of~\eqref{local:mixed}.
Pick the same function in the second equation of~\eqref{local-problem:DN} and deduce
\[
\begin{split}
\int_\E \div(\sigmaboldn) \vn 	& = \sum_{\nu \in \mathcal V^\E} \int_\E \div(\sigmaboldnnu) \vn \\
						& = \sum_{\nu \in \mathcal V^\E} \Bigg\{ (\Pitildez \phitildenu, \f \vn)_{0,\E} - ( \nabla \Pinablap \phitildenu, \nabla \Pinablap \un \vn)_{0,\E}  \\
						& \quad \quad \quad \quad-\frac{1}{\vert \E \vert} ( \StildeE( (I-\Pinablap)\phitildenu,  (I-\Pinablap) \utilden ) , \vn)_{0,\E}   \Bigg \}.
\end{split}
\]
Using that
\[
\Pitildez \left(\sum_{\nu \in \mathcal V ^\E} \phitildenu  \right) = 1, \quad \quad  \nabla \Pinablap\left(\sum _{\nu \in \mathcal V^\E} \phitildenu \right) =0, \quad \quad  (I-\Pinablap) \left(\sum _{\nu \in \mathcal V^\E} \phitildenu \right) = 0,
\]
we get
\[
\begin{split}
\int_\E \div(\sigmaboldn) \vn 	& = \int_\E \f \vn - \sum_{\nu \in \mathcal V^\E} \left\{ \frac{1}{\vert \E \vert} \left(  ( \StildeE( (I-\Pinablap)\phitildenu,  (I-\Pinablap) \utilden ) , \vn)_{0,\E}  \right)    \right\}\\
									& = \int_\E \f \vn - \frac{(1,\vn)_{0,\E}}{\vert \E \vert} \sum_{\nu \in \mathcal V^\E}  \StildeE( (I-\Pinablap)\phitildenu,  (I-\Pinablap) \utilden ) = \int_\E \f \vn, \\
\end{split}
\]
which is the assertion.
\end{proof}

\subsection{A new error estimator and its reliability} \label{subsection:reliability-local}
In this section, we show a result proving the reliability of an error estimator computed by means of the function~$\sigmaboldn$~\eqref{global:sigma} as well as of the solution to the primal discrete formulation~\eqref{VEM:primal}.
This is the virtual element counterpart of the classical counterpart by Prager and Synge~\cite{prager1947approximations}.

For all~$\E \in \taun$, introduce the local flux reconstruction error estimators
\[
\begin{split}
\etalocE ^2 	& = \SE((\Ibold-\Piboldzp)\sigmaboldn , (\Ibold - \Piboldzp) \sigmaboldn) + \StildeE( (I -\Pinablap)\utilden , (I -\Pinablap)\utilden  ) \\
				& \quad + \Vert \Piboldzp \sigmaboldn + \nabla \Pinablap \utilden \Vert_{0,\Omega}^2.
\end{split}
\]
We define the global local flux reconstruction error estimator~$\etaloc$ as
\begin{equation} \label{global:localized:error-estimator}
\etaloc^2 = \sum_{\E \in \taun} \etalocE^2.
\end{equation}

\begin{thm} \label{theorem:h-reliability}
Let the assumptions (\textbf{G1}), (\textbf{G2}), (\textbf{K}), and (\textbf{D}) be valid.
Let~$\utilde$ and~$\utilden$ be the solutions to~\eqref{primal:formulation} and~\eqref{VEM:primal}, respectively, and~$\sigmaboldn$ be defined as in~\eqref{global:sigma}.
The following upper bound on the error of the primal formulation is valid:
\begin{equation} \label{reliability:bound-h}
\begin{split}
\vert \utilde - \utilden \vert_{1,\Omega}^2 	& \lesssim  \left[ \max_{\E \in \taun} (\max(\alpha_*^{-1},\alphatilde_*^{-1})) \right] \etaloc^2 + \sum_{\E \in \taun} \left(  \frac{\hE^2}{\p^2} \Vert  \f - \div \sigmaboldn \Vert_{0,\E}^2 \right).  \\
\end{split}
\end{equation}
The hidden constant is independent of~$\p$.
\end{thm}
\begin{proof}
For all~$v\in H^1_0(\Omega)$, we have
\begin{equation} \label{big:bound:h}
\begin{split}
\vert \utilde - \utilden \vert_{1,\Omega}	& = \sup_{v\in H^1_0(\Omega),\, \vert v \vert_{1,\Omega}=1} (\nabla (\utilde - \utilden), \nabla v)_{0,\Omega} \overset{\eqref{primal:formulation}}{=} \sup_v \{ (\f,v)_{0,\Omega} - (\nabla \utilden, \nabla v)_{0,\Omega} \} \\
								& = \sup_v \{ (\f,v)_{0,\Omega} - (\nabla \utilden, \nabla v)_{0,\Omega} - (\sigmaboldn , \nabla v)_{0,\Omega} - (\div \sigmaboldn, v)_{0,\Omega}     \}\\
								& = \sup_v \{ (\f-\div \sigmaboldn, v)_{0,\Omega}  -  (\sigmaboldn + \nabla \utilden, \nabla v)_{0,\Omega}    \}\\
								& =: \sup_v \vert A - B \vert \le \sup_v \{ \vert A \vert + \vert B \vert   \}.\\
\end{split}
\end{equation}
We prove an upper bound on the two terms on the right-hand side of~\eqref{big:bound:h} separately. We begin with the first one.
Using Lemma~\ref{lemma:p-version:div-property} testing with~$\vn$, the piecewise~$L^2$ projector onto~$\mathbb P_{\p-1}(\E)$,
using $\h\p$-best polynomial approximation properties, and the $\ell^2$ Cauchy-Schwarz inequality, we deduce
\begin{equation} \label{bound:term-A}
\begin{split}
A 	& =  (\f - \div\sigmaboldn, v)_{0,\Omega} \overset{\eqref{higher-moments:divergence}}{=} \sum_{\E \in \taun} (\f - \div\sigmaboldn, v - \vn)_{0,\E} \\
	& \lesssim \left( \sum_{\E \in \taun} \frac{\hE^2}{\p^2} \Vert \f - \div\sigmaboldn \Vert^2_{0,\E}     \right)^{\frac{1}{2}}  \vert v\vert_{1,\Omega}. \\
\end{split}
\end{equation}
As for the upper bound on the second term on the right-hand side of~\eqref{big:bound:h}, we observe that
\[
\begin{split}
\vert B \vert	& = \vert (\sigmaboldn + \nabla \utilden, \nabla v)_{0,\Omega} \vert \\
			& \le \vert (\sigmaboldn - \Piboldzp \sigmaboldn, \nabla v)_{0,\Omega} \vert  +  \vert (\Piboldzp \sigmaboldn + \nabla \Pinablap\utilden, \nabla v)_{0,\Omega} \vert + \vert ( \nabla \utilden - \nabla \Pinablap\utilden, \nabla v)_{0,\Omega} \vert \\
			& \le \left(  \Vert \sigmaboldn - \Piboldzp \sigmaboldn \Vert_{0, \Omega} + \Vert \Piboldzp \sigmaboldn + \nabla \Pinablap \utilden \Vert_{0,\Omega}     + \Vert  \nabla \utilden - \nabla \Pinablap\utilden   \Vert_{0,\Omega}   \right) \vert v \vert_{1,\Omega} .\\
\end{split}
\]
Using the coercivity property of the stabilizations in~\eqref{local_stab:primal} and~\eqref{local_stab:mixed}, we get
\begin{equation} \label{bound:term-B}
\begin{split}
\vert B \vert
& \le \Bigg( \sum_{\E \in \taun} \left[ \alpha_*^{-1} \SE((\Ibold- \Piboldzp) \sigmaboldn, (\Ibold- \Piboldzp) \sigmaboldn) + \alphatilde_*^{-1}  \StildeE ( (I-\Pinablap) \utilden, (I-\Pinablap) \utilden ) \right] ^{\frac{1}{2}}\\
& \quad\quad\quad\quad\quad + \Vert \Piboldzp \sigmaboldn +  \nabla \Pinablap\utilden \Vert_{0,\Omega} \Bigg) \vert v \vert_{1,\Omega} .\\
\end{split}
\end{equation}
The assertion follows by plugging~\eqref{bound:term-A} and~\eqref{bound:term-B} in~\eqref{big:bound:h}.
\end{proof}
Theorem~\ref{theorem:h-reliability} shows the reliability of the local flux reconstruction error estimator.
The second term on the right-hand side of~\eqref{reliability:bound-h} represents the oscillation in the equilibrated flux condition~\eqref{higher-moments:divergence}.

\subsection{Lack of efficiency} \label{subsection:no-efficiency}
In this section, we give hints on why we are not able to prove the efficiency of the error estimator~\eqref{global:localized:error-estimator} and why it is probably not valid in general.

To this aim, we first recall what happens in the finite element setting.
Given the patch~$\omeganu$ consisting of triangles around an internal vertex~$\nu$, let~$\RT(\omeganu)$ be the associated Raviart-Thomas space of order~$\p$.
Moreover, let~$\Qn(\omeganu)$ be the space of piecewise polynomials of degree~$\p$ associated with the subtriangulation of~$\omeganu$.
Observe that the partition of unity elements~$\phitildenu$ are standard linear hat functions.
The FEM counterpart of~\eqref{local:mixed} reads
\begin{equation} \label{local:mixed:FEM}
\begin{cases}
\text{find } (\sigmaboldnnu, \rnnu) \in \RT(\omeganu) \times \Qn(\omeganu) \text{ such that, for all } \tauboldnnu \in \RT(\omeganu) \text{ and } \qnnu \in \Qn(\omeganu), \\
(\sigmaboldnnu, \tauboldnnu)_{0,\omeganu} - (\div (\tauboldnnu), \rnnu)_{0,\omeganu} = - (\phitildenu \, \nabla \utilden, \tauboldnnu)_{0, \omeganu} \\
- (\div \sigmaboldnnu, \qnnu)_{0,\omeganu} = - (\phitildenu \, \f - \nabla \phitildenu \cdot \nabla \utilden , \qnnu)_{0, \omeganu} . \\
\end{cases}
\end{equation}
The first step towards the proof of the local efficiency of the error estimator is based on an~$H^1$ representation of the residual. In particular, introduce~$\rnu \in H^1(\omeganu) / \mathbb R$ the solution to
\begin{equation} \label{H1-residual:FEM}
(\nabla \rnu, \nabla v)_{0,\omeganu} = -(\phitildenu \nabla\utilden, \nabla v) + (\phitildenu f - \nabla \phitildenu \cdot \nabla \utilden, v)_{0,\omeganu} \quad \forall v \in H^1(\omeganu) / \mathbb R.
\end{equation}
This equation is obtained by formally substituting~$v$ and~$\nabla v$ to~$\qnnu$ and~$\tauboldnnu$, using twice an integration by parts, and combining the two equations in~\eqref{local:mixed:FEM}.

As in~\cite{carstensen1999fully, braess2009equilibrated}, the following bound can be proven:
\begin{equation} \label{1st-step-efficiency}
\vert \rnu \vert_{1,\omeganu} \lesssim \vert u - \utilden \vert _{1,\omeganu},
\end{equation}
where the hidden constant depends on a Poincar\'e constant.

A crucial ingredient in the proof of~\eqref{1st-step-efficiency} consists in observing that
\begin{equation} \label{remark:Braess-Schoeberl}
\begin{split}
-(\phitildenu \nabla\utilden, \nabla v) + (\phitildenu f - \nabla \phitildenu \cdot \nabla \utilden, v)_{0,\omeganu} 
& = (\f, \phitildenu v)_{0,\omeganu} - (\nabla \utilden, \nabla (\phitildenu \, v))_{0,\omeganu} \\
& = (\nabla(u-\utilden), \nabla (\phitildenu\, v))_{0,\omeganu}.
\end{split}
\end{equation}
This is the first step towards the proof of the efficiency, as then it suffices to show that the error estimator is smaller than the energy of the $H^1$ residual.
In what follows, we explain how the VEM counterpart of~\eqref{H1-residual:FEM} looks like, and what are the issues in deriving estimates of the type in~\eqref{1st-step-efficiency}.
\medskip 

The main difference in the FE and VE approaches is that in the latter we have to deal with projectors, stabilizations, and functions that are not available in closed-form.
We formally substitute~$v$ and~$\nabla v$ to~$\qnnu$ and~$\tauboldnnu$. Using the first equation in~\eqref{local:mixed}, we can write
\begin{equation} \label{first-step-no:efficiency}
\begin{split}
& (\nabla \rnu ,\nabla v)_{0,\omeganu}\\
& = -\sum_{\E \in \taunomeganu} \left[ (\Pitildez \phitildenu \, \nabla \Pinablap \utilden, \Piboldzp \nabla v)_{0,\E} - \anE(\sigmaboldnnu, \nabla v)  \right]\\
& = -\sum_{\E \in \taunomeganu} \left[ (\Pitildez \phitildenu \, \nabla \Pinablap \utilden, \Piboldzp \nabla v)_{0,\E} \underbrace{- \anE(\sigmaboldnnu, \nabla v) + \aE(\sigmaboldnnu, \nabla v)}_{=:A} - \aE(\sigmaboldnnu, \nabla v) \right].\\
\end{split}
\end{equation}
The term~$A$ is a consistency error, which can be approximated with optimal rate.
Thence, we deal with the last term appearing on the right-hand side of~\eqref{first-step-no:efficiency}.
Using the definition of orthogonal projectors, an integration by parts, and the definition of the jumps, we get
\[
- \aomeganu(\sigmaboldnnu, \nabla v)  = (\div \sigmaboldnnu, v)_{0,\omeganu}.
\]
Inserting this bound into~\eqref{first-step-no:efficiency} and using formally the second equation in~\eqref{local:mixed} yield
\[
\begin{split}
& (\nabla \rnu,\nabla v)_{0,\omeganu}\\
& = A  + \sum_{\E \in \taunomeganu} \left[ -(\Pitildez \phitildenu \, \nabla \Pinablap \utilden, \Piboldzp \nabla v)_{0,\E} + (\Pitildez \phitildenu \, \f - \nabla \Pinablap \phitildenu \cdot \nabla \Pinablap \utilden,v)_{0,\E} \right].\\
& = A \underbrace{-\sum_{\E \in \taun} (\nabla \Pinablap \utilden, \Pitildez \phitildenu\, \Piboldzp \nabla v + \nabla \Pinablap \phitildenu \, v)_{0,\E}}_{=:B}
																+ \sum_{\E \in \taun} (\f, \Pitildez \phitildenu v)_{0,\E}.
\end{split}
\]
Eventually, we analyze the last term on the right-hand side. Using the continuous problem~\eqref{strong:primal} and an integration by parts lead us to
\[
\begin{split}
\sum_{\E \in \taun} (\f, \Pitildez \phitildenu v)_{0,\E} 	& = \sum_{\E \in \taunomeganu} (\nabla \utilde, \nabla (\Pitildez \phitildenu\, v))_{0,\E} - \sum_{\e \in \Enomeganu} (u, \llbracket \Pitildez \phitildenu\, v \rrbracket)_{0,\e}=:C + D.
\end{split}
\]
Combining all the above estimates, we arrive at
\[
(\nabla \rnu,\nabla v)_{0,\omeganu} = A+B+C+D.
\]
In the FE setting the terms~$B$ and~$C$ are combined together, see~\eqref{remark:Braess-Schoeberl}.
Here, the presence of different projectors forbids us to use either the product rule or the continuous formulation as in~\eqref{remark:Braess-Schoeberl},
which renders the simplification resulted by combining the partition of unity impossible.

In Section~\ref{section:nr:local} below, we provide numerical evidence that the efficiency does not take place in the current setting.

\section{Numerical results on the local flux reconstruction} \label{section:nr:local}
In this section, we present some numerical results on the local flux reconstruction presented in Section~\ref{section:localization-mixed}.
The aim is to show that reliability~\eqref{reliability:bound-h} and the equilibrated flux condition~\eqref{higher-moments:divergence} are valid.
Interestingly, we shall observe that the efficiency does not take place for the high-order VEM.

We consider the following test case.
\paragraph*{Test case~$3$.}
Consider the square domain~$\Omega_3=(0,1)^2$ and the exact solution
\begin{equation} \label{u3}
u_3(x,y) = x(1-x)y(1-y).
\end{equation}
The primal formulation of the problem we are interested in is such that we have: zero Dirichlet boundary conditions on the boundary edges;
$\k=1$; a right-hand side computed accordingly with~\eqref{u3}.
\medskip

We want to analyse the behaviour of the following quantities:
\begin{itemize}
\item the error of the method~$\vert u - \un \vert_{1,\Omega}$;
\item the error estimator~$\etaloc$~\eqref{global:localized:error-estimator};
\item the oscillation in the equilibrated flux condition, given by
\begin{equation} \label{oscillation:flux-condition}
\sqrt{\sum_{\E \in \taun} \left(  \frac{\hE^2}{\p^2} \Vert  \f - \div\sigmaboldn \Vert_{0,\E}^2 \right)};
\end{equation}
\item the quantity
\begin{equation} \label{new:quantity}
\sqrt{\sum_{\nu \in \Nun} \Vert  \Piboldzp \sigmaboldnnu + \Pitildez \phitildenu \nabla \Pinablap \utilden \Vert_{0,\omeganu}^2}.
\end{equation}
\end{itemize}
We are interested in the performance of the $\h$-version of the method for some values of degree of accuracy~$\p$.
More precisely, in Section~\ref{subsection:nr-local-low}, we present the case~$\p=1$ where we shall also observe efficiency.
Instead, in Section~\ref{subsection:nr-local-high}, we take~$\p=2$ and show that we miss efficiency, albeit the equilibration of fluxes is valid in agreement with the theoretical prediction of Lemma~\ref{lemma:p-version:div-property}.

\subsection{The case~$\p=1$} \label{subsection:nr-local-low}
In this section, we consider the test case~$3$ with exact solution~$u_3$ in~\eqref{u3}
and study the performance of the $\h$-version of the method with degree of accuracy~$\p=1$ using uniform triangular and Cartesian meshes.
The local problems~\eqref{local:mixed} and~\eqref{local-problem:DN} are extremely simplified for the lowest order case on triangular meshes.
Indeed, the virtual element spaces with~$\p=1$ on triangular meshes  are standard finite element methods.
Therefore, all the projectors and stabilizations disappear in the formulation.

In Tables~\ref{table:nr-local-low} and~\ref{table:nr-local-low-Cartesian}, we depict the decay of the error of the method~$\vert u - \un \vert_{1,\Omega}$ with the corresponding
experimentally determined order of convergence (EOC), the error estimator~$\etaloc$~\eqref{global:localized:error-estimator},
the oscillation in the equilibrated flux condition introduced in~\eqref{oscillation:flux-condition}, and quantity~\eqref{new:quantity}, for uniform triangular and Cartesian meshes, respectively.

\begin{table}[H] 
\centering
\begin{tabular}{|c|cc|cc|c|c|}
\hline
 	& $\vert u - \un \vert_{1,\Omega}$ 	& EOC 			& $\etaloc$ 		& $\etaloc / \vert u - \un \vert_{1,\Omega}$	& \eqref{oscillation:flux-condition}  	& \eqref{new:quantity}\\
\hline
mesh~1 	& 0.446 								& --- 				& 0.469 			& 1.0515												& 6.454e-02								& 7.784e-02 \\
mesh~2 	& 0.234 								& 0.927			& 0.260 			& 1.1121												& 1.638e-02 							& 5.860e-02 \\
mesh~3 	& 0.120 								& 0.966			& 0.137 			& 1.1422												& 4.112e-03 							& 5.104e-02 \\
mesh~4 	& 0.060								& 0.9872			& 0.071 			& 1.1804												& 1.029-03 								& 4.869e-02 \\
\hline
\end{tabular}
\caption{$\h$-version of the method using triangular meshes with degree of accuracy~$\p=1$.}
\label{table:nr-local-low}
\end{table}

\begin{table}[H] 
\centering
\begin{tabular}{|c|cc|cc|c|c|}
\hline
	 		& $\vert u - \un \vert_{1,\Omega}$ 	& EOC 			& $\etaloc$ 		& $\etaloc / \vert u - \un \vert_{1,\Omega}$	& \eqref{oscillation:flux-condition}  	& \eqref{new:quantity}\\
\hline
mesh~1 	& 0.339 										& --- 				& 0.362 			& 1.068												& 1.035e-02								& 5.723e-02 \\
mesh~2 	& 0.169 										& 0.999			& 0.186		 	& 1.098												& 1.638e-02								& 4.703e-02 \\
mesh~3 	& 0.084 										& 0.999			& 0.093		 	& 1.106												& 2.600e-03 							& 4.407e-02 \\
mesh~4 	& 0.042										& 0.999 			& 0.046 			& 1.108												& 6.507e-04 							& 4.329e-02 \\
\hline
\end{tabular}
\caption{$\h$-version of the method using uniform Cartesian meshes with degree of accuracy~$\p=1$.}
\label{table:nr-local-low-Cartesian}
\end{table}

From Table~\ref{table:nr-local-low}, we observe that the decay of the error is optimal and that the error estimator~$\etaloc$ is efficient.
This could have been expected, since on triangular meshes and~$\p=1$ the VEM coincides with the lowest order FEM where it is well-known that the error estimator is reliable and efficient.
Moreover, the oscillation in the equilibrated flux condition introduced in~\eqref{oscillation:flux-condition} is of higher order.

\subsection{The case~$\p=2$} \label{subsection:nr-local-high}
In this section, we consider the test case~$3$ with exact solution~$u_3$ in~\eqref{u3}
and study the performance of the $\h$-version of the method with degree of accuracy~$\p=2$ using uniform triangular and Cartesian meshes.

In Tables~\ref{table:nr-local-high} and~\ref{table:nr-local-high-Cartesian}, we depict the decay of the error of the method~$\vert u - \un \vert_{1,\Omega}$ with the corresponding EOC, 
the error estimator~$\etaloc$~\eqref{global:localized:error-estimator},
the oscillation in the equilibrated flux condition introduced in~\eqref{oscillation:flux-condition}, and quantity~\eqref{new:quantity}, for uniform triangular and Cartesian meshes, respectively.

\begin{table}[H] 
\centering
\begin{tabular}{|c|cc|cc|c|c|}
\hline
	 		& $\vert u - \un \vert_{1,\Omega}$ 	& EOC 		& $\etaloc$ 		& $\etaloc / \vert u - \un \vert_{1,\Omega} $	& \eqref{oscillation:flux-condition} 	& \eqref{new:quantity}		\\
\hline
mesh~1 	& 0.095 										& --- 			& 0.171 			& 1.803												& 5.892e-03 					& 5.105e-02 \\
mesh~2 	& 0.024 										& 1.989		& 0.089 			& 3.718												& 7.365e-04  					& 4.722e-02 \\
mesh~3 	& 0.006 										& 1.991		& 0.064 			& 10.649												& 9.207e-05  					& 4.614e-02 \\
mesh~4 	& 0.001										& 2		 	& 0.055		 	& 36.798												& 1.150e-05 					& 4.588e-02 \\
\hline
\end{tabular}
\caption{$\h$-version of the method using triangular meshes with degree of accuracy~$\p=2$.}
\label{table:nr-local-high}
\end{table}

\begin{table}[H] 
\centering
\begin{tabular}{|c|cc|cc|c|c|}
\hline
	 		& $\vert u - \un \vert_{1,\Omega}$ 	& EOC 		& $\etaloc$ 	& $\etaloc / \vert u - \un \vert_{1,\Omega} $	& \eqref{oscillation:flux-condition} 	& \eqref{new:quantity}		\\
\hline
mesh~1 	& 0.079 										& --- 			& 0.162	 	& 2.053												& 2.329e-03 							& 5.518e-02 \\
mesh~2 	& 0.020 										& 1.965		& 0.070	 	& 3.461												& 2.911e-04  							& 5.130e-02 \\
mesh~3 	& 0.005										& 1.991		& 0.033		& 6.529												& 3.639e-05  							& 5.027e-02 \\
mesh~4 	& 0.001										& 1.998	 	& 0.016		& 12.846												& 4.549e-06							 	& 5.000e-02 \\
\hline
\end{tabular}
\caption{$\h$-version of the method using uniform Cartesian meshes with degree of accuracy~$\p=2$.}
\label{table:nr-local-high-Cartesian}
\end{table}

Differently from the low order case, in Tables~\ref{table:nr-local-high} and~\ref{table:nr-local-high-Cartesian}, we observe a loss of efficiency of the method.
This is in agreement with the arguments detailed in Section~\ref{subsection:no-efficiency}.
Interestingly enough, in agreement with Section~\ref{subsection:reliability-local}, the error estimator is however reliable and the oscillation in the equilibrated fluxes decay at high-order for both meshes.

\subsection{The low order adaptive scheme} \label{subsection:nr-local-adaptive}
In this section, we consider the test case~$3$ with exact solution~$u_3$ in~\eqref{u3}
and study the performance of the $\h$-adaptive method with degree of accuracy~$\p=1$ using a starting coarse triangular mesh.
In Table~\ref{table:nr-local-adaptive}, we depict the decay of the error of the method~$\vert u - \un \vert_{1,\Omega}$, the error estimator~$\etaloc$~\eqref{global:localized:error-estimator},
the oscillation in the equilibrated flux condition introduced in~\eqref{oscillation:flux-condition}, and quantity~\eqref{new:quantity}.

\begin{table}[H] 
\centering
\begin{tabular}{|c|c|cc|c|c|}
\hline
	 			& $\vert u - \un \vert_{1,\Omega}$ 	& $\etaloc$ 	& $\etaloc / \vert u - \un \vert_{1,\Omega} $	& \eqref{oscillation:flux-condition} 	& \eqref{new:quantity}		\\
\hline
iteration~1 	& 0.446										& 0.469	 	& 1.051												& 6.454e-02 							& 7.784e-02 \\
iteration~2 	& 0.385 										& 0.534 		& 1.386												& 3.951e-02  							& 7.813e-02 \\
iteration~3 	& 0.234 										& 0.260 		& 1.112												& 1.638e-02  							& 5.860e-02 \\
iteration~4 	& 0.182										& 0.239 		& 1.313												& 1.178e-02 							& 6.102e-02 \\
iteration~5 	& 0.120										& 0.137 		& 1.142												& 4.112e-03 							& 5.104e-02 \\
iteration~6 	& 0.098										& 0.162 		& 1.652												& 2.998e-03								& 5.517e-02 \\
iteration~7 	& 0.059										& 0.092		& 1.555												& 1.001e-03								& 5.045e-02 \\
iteration~8 	& 0.049										& 0.112	 	& 2.283												& 7.417e-04 								& 5.238e-02 \\
\hline
\end{tabular}
\caption{$\h$-adaptive version of the method using a starting coarse triangular mesh} with degree of accuracy~$\p=1$.
\label{table:nr-local-adaptive}
\end{table}

From Table~\ref{table:nr-local-adaptive}, we observe a loss of efficiency for the error estimator. This is clear, owing to the fact that the adaptive refinement strategy produces nontriangular meshes, even starting with a uniform triangular mesh.
Indeed, hanging nodes are added while performing the mesh refinement. Notwithstanding, we can still appreciate the reliability of the error estimator as well as the decay of the oscillation in the equilibrated fluxes.

\subsection{A global residual equilibration of fluxes}
\label{subsection:nr-global-local}
In Sections~\ref{subsection:nr-local-low} and~\ref{subsection:nr-local-high}, we observed that the error estimator is reliable,
in agreement with the theoretical results of Section~\ref{subsection:reliability-local}.
However, as motivated in Section~\ref{subsection:no-efficiency}, we observed a lack of efficiency.
To investigate better the effect of the mixed projectors appearing on the right-hand side of~\eqref{local:mixed}, we consider the ``global version'' of the localized problem
\begin{equation} \label{global-local:mixed}
\begin{cases}
\text{find } (\sigmaboldnnu, \rnnu) \in \Sigmaboldn \times \Vn \text{ such that, for all } \tauboldnnu \in \Sigmaboldn \text{ and } \qnnu \in \Vn, \\
\an (\sigmaboldn, \tauboldn) - (\div (\tauboldn), \rn)_{0,\omeganu} = -  ( \nabla \Pinablap \utilden, \Piboldzp \tauboldn)_{0, \Omega} \\
- (\div \sigmaboldn, \qn)_{0,\omeganu} = - ( \f  , \qnnu)_{0, \Omega} \\
\end{cases}
\end{equation}
and compute the quantities
\begin{equation} \label{quantities:global-local}
\Vert \rn \Vert_{0,\Omega},\quad \quad \quad \quad \quad  \Vert \Piboldzp \sigmaboldn    +   \nabla \Pinablap \un    \Vert _{0,\Omega}.
\end{equation}
We expect that the two quantities~\eqref{quantities:global-local} converge to zero. The first one represents the global residual, whereas the second one is the error on the projected flux reconstruction.
For the test case~$3$, we present the numerical results for the $\h$-version of the method with uniform Cartesian meshes for degree~$\p=2$ and~$3$ in Table~\ref{table:global-local}.

\begin{table}[H] 
\centering
\begin{tabular}{|c|c|c|c|c|}
\hline
 				& $\p=2$ , $\Vert \rn \Vert_{0,\Omega}$	& $\Vert \Piboldzp \sigmaboldn +\nabla \Pinablap \un    \Vert _{0,\Omega}$ 	& $\p=3$, $\Vert \rn \Vert_{0,\Omega}$ & $\Vert \Piboldzp \sigmaboldn +\nabla \Pinablap \un \Vert _{0,\Omega}$\\
\hline
mesh~1 		& 1.101e-02											& 2.321e-02																					& 1.766e-04										& 4.256e-04 	\\
mesh~2 		& 2.103e-03 											& 3.920e-03																					& 1.081e-05										& 2.896e-05  	\\
mesh~3 		& 7.213e-04										& 6.695e-04																					& 6.726e-07										& 1.913e-06  	\\
mesh~4 		& 1.802e-04										& 1.166e-04 																				& 4.201e-08										& 1.308e-07 	\\
mesh~5		& 4.505e-05										& 2.076e-05																				& 2.626e-09										& 9.432e-09\\
\hline
\end{tabular}
\caption{$\h$-version of the global version~\eqref{global-local:mixed} using a starting coarse Cartesian mesh} with degrees of accuracy~$\p=2$ and~$\p=3$.
\label{table:global-local}
\end{table}
Differently from the pure localized problem, the accuracy increases when increasing the order of the method.
This is an additional confirmation that the main culprit for the lack of efficiency must be sought in the combination of different projection operators in the right-hand side of~\eqref{local:mixed}.

\section{Conclusion} \label{section:conclusion}
We presented an a posteriori error analysis for the virtual element method based on equilibrated fluxes.
We introduced an equilibrated reliable and efficient a posteriori error estimator, using the virtual element solutions to the primal and mixed formulations.
Additionally, we showed that the discrete inf-sup constant for the mixed VEM is $\p$-independent
and constructed an explicit stabilization for the mixed VEM, characterized by lower and upper bounds with explicit dependence on the (local) degree of accuracy.

Several numerical experiments have been illustrated.
On the one hand, we showed that the effectivity index for the hypercircle method is $\p$-independent in practice. This is a major improvement with respect to the residual error estimator case.
On the other hand, we discussed an $\h\p$-adaptive algorithm and applied it to two test cases.
We observed exponential convergence in terms of the cubic root of the number of degrees of freedom in all the test cases.

Eventually, we began the analysis of the localized flux reconstruction in VEM.
Notably, we introduced a reliable computable error estimator, which can be obtained using the solution to the primal formulation and a combination of solutions to local mixed problems.
Numerics showed that the equilibrium condition is fulfilled but efficiency does not occur with the exception of the low-order case.
We also gave theoretical justifications for this lack of efficiency. It turns out that the culprit has to be sought in the mismatch in the use of the projection operators appearing in the proposed local mixed VEM.

\paragraph*{Acknowledgements.}
L. M. acknowledges the support of the Austrian Science Fund (FWF) project P33477. 

{\footnotesize
\bibliography{bibliogr}
}
\bibliographystyle{plain}

\end{document}